\documentclass[a4paper,11pt,twoside]{amsart}
\usepackage[english]{babel}
\usepackage[utf8]{inputenc}

\usepackage[a4paper,inner=3cm,outer=3cm,top=4cm,bottom=4cm,pdftex]{geometry}
\usepackage{fancyhdr}
\usepackage{titlesec}
\usepackage{color}
\usepackage{bold-extra}
\usepackage{ mathrsfs }
\titleformat{\section}{\normalfont\scshape\centering}{\thesection}{1em}{}
  \titleformat{\subsection}{\bfseries}{\thesubsection}{1em}{}

\usepackage{comment}
\usepackage{graphics}
\usepackage{aliascnt}
\usepackage[pdftex,citecolor=green,linkcolor=red]{hyperref}

\usepackage{enumerate}
\usepackage{amsmath}
\usepackage{amsfonts}
\usepackage{amssymb}
\usepackage{amsthm}
\usepackage{comment}
\usepackage{mathtools}

\newtheorem{theorem}{Theorem}[section]
\newtheorem{corollary}[theorem]{Corollary}

\newtheorem{lemma}[theorem]{Lemma}
\newtheorem{proposition}[theorem]{Proposition}
\theoremstyle{definition}
\newtheorem{definition}[theorem]{Definition}
\newtheorem{remark}[theorem]{Remark}
\newtheorem{conjecture}[theorem]{Conjecture}

\numberwithin{equation}{section}

\newcommand\eps{\varepsilon}

\renewcommand{\Re}{\textnormal{Re}}

\newcommand\E{\mathbb{E}}
\newcommand\R{\mathbb{R}}
\newcommand\Z{\mathbb{Z}}
\newcommand\N{\mathbb{N}}
\newcommand\C{\mathbb{C}}

\newcommand\Siegel{{\textnormal{Siegel}}}

\newcommand{\asum}{\sideset{}{^{\ast}}\sum}
\begin{document}
\title[Hardy--Littlewood--Chowla with a Siegel zero]{The Hardy--Littlewood--Chowla conjecture in the presence of a Siegel zero}

\author[Terence Tao]{Terence Tao}
\address{Department of Mathematics, University of California, Los Angeles, CA 90095-1555,USA}
\email{tao@math.ucla.edu}

\author[Joni Ter\"{a}v\"{a}inen]{Joni Ter\"{a}v\"{a}inen}
\address{Mathematical Institute, University of Oxford, Oxford, UK}
\address{Department of Mathematics and Statistics, University of Turku, Turku, Finland}
\email{joni.p.teravainen@gmail.com}

\begin{abstract} Assuming that Siegel zeros exist, we prove a hybrid version of the Chowla and Hardy--Littlewood prime tuples conjectures. Thus, for an infinite sequence of natural numbers $x$, and any distinct integers $h_1,\dots,h_k,h'_1,\dots,h'_\ell$, we establish an asymptotic formula for
$$\sum_{n\leq x}\Lambda(n+h_1)\cdots \Lambda(n+h_k)\lambda(n+h_{1}')\cdots \lambda(n+h_{\ell}')$$
for any $0\leq k\leq 2$ and $\ell \geq 0$. Specializing to either $\ell=0$ or $k=0$, we deduce the previously known results on the Hardy--Littlewood (or twin primes) conjecture and the Chowla conjecture under the existence of Siegel zeros, due to Heath-Brown and Chinis, respectively. The range of validity of our asymptotic formula is wider than in these previous results.
\end{abstract}


\maketitle

\section{Introduction}

\subsection{The Hardy--Littlewood--Chowla conjecture and Siegel zeroes}

Let $\lambda \colon \N \to \{-1,+1\}$ denote the Liouville function.  We have the following well known conjecture of Chowla~\cite{chowla}:

\begin{conjecture}[Chowla's conjecture]\label{chowla-conj}  Let $h'_1,\dots,h'_\ell$ be distinct fixed natural numbers for some fixed $\ell \geq 1$.  Then\footnote{See Section~\ref{notation-sec} for our conventions on asymptotic notation.}
$$ \E_{n \leq x} \lambda(n+h'_1) \cdots \lambda(n+h'_\ell) = o(1)$$
as $x \to \infty$.  
\end{conjecture}

Here and in the sequel, $n$ is understood to range over natural numbers, and we use the averaging notation $\E_{n \in A} f(n) \coloneqq \frac{1}{|A|} \sum_{n \in A} f(n)$ for any set $A$ of a finite cardinality $|A|$.  The reasons for the primes in the notation $h'_1,\dots,h'_\ell$ is for compatibility with Conjecture~\ref{hlc} below.

For $\ell=1$ Chowla's conjecture is equivalent to the prime number theorem, but the conjecture is open for all $\ell \geq 2$, although a slightly weaker ``logarithmically averaged'' conjecture is known to hold for $\ell=2$~\cite{tao-chowla} or for odd $\ell$~\cite{tt-odd, tt-structure}.  All the discussion here concerning the Liouville function $\lambda$ has a counterpart for the M\"obius function $\mu$, but for simplicity of exposition we restrict attention to the Liouville function here.

The analogous conjecture for the von Mangoldt function $\Lambda \colon \N \to \R^+$ is the well known \emph{prime tuples conjecture} of Hardy and Littlewood~\cite{hardy-littlewood}:

\begin{conjecture}[Hardy--Littlewood prime tuples conjecture]\label{hl-conj}   Let $h_1,\dots,h_k$ be distinct fixed natural numbers for some fixed $k \geq 0$.  Then
$$ \E_{n \leq x} \Lambda(n+h_1) \cdots \Lambda(n+h_k) = {\mathfrak S} + o(1)$$
as $x \to \infty$, where the \emph{singular series} ${\mathfrak S}$ is defined by the formula
\begin{equation}\label{ss-def}
 {\mathfrak S} \coloneqq \prod_p \beta_p,
\end{equation}
the local factors $\beta_p$ are defined by
\begin{equation}\label{betap-def}
\beta_p \coloneqq \E_{n \in \Z/p\Z} \Lambda_p(n+h_1) \cdots \Lambda_p(n+h_k)=\left(1-\frac{1}{p}\right)^{-k}\left(1-\frac{|\{h_1,\ldots, h_k\}\pmod p|}{p}\right)
\end{equation}
and $\Lambda_p \colon \Z/p\Z \to \R^+$ is the local von Mangoldt function $\Lambda_p(n) \coloneqq \frac{p}{p-1} 1_{n \neq 0\ (p)}$.  (In this paper we adopt the convention that the empty product is equal to $1$.)
\end{conjecture}

It is not difficult to show the asymptotic
\begin{equation}\label{betap-asym}
 \beta_p = 1 + O\left( \frac{1}{p^2} \right),
\end{equation}
so the product in~\eqref{ss-def} converges, though it could vanish if the $h_1,\dots,h_k$ cover a complete set of residues modulo $p$ for some prime $p$.  Conjecture~\ref{hl-conj} is trivial for $k=0$ and equivalent to the prime number theorem for $k=1$, but is open for all other values of $k$, with the $k=2$ case already implying the notorious twin prime conjecture.

It is natural to unify Conjecture~\ref{chowla-conj} and Conjecture~\ref{hl-conj} as follows.

\begin{conjecture}[Hardy--Littlewood--Chowla conjecture]\label{hlc}  Let $k, \ell \geq 0$, and let $h_1,\dots,h_k$, $h'_1,\dots,h'_\ell$ be distinct fixed natural numbers.  Then
$$ \E_{n \leq x} \Lambda(n+h_1) \cdots \Lambda(n+h_k) \lambda(n+h'_1) \cdots \lambda(n+h'_\ell) = {\mathfrak S} + o(1)$$
as $x \to \infty$, where ${\mathfrak S}$ is defined by~\eqref{ss-def} when $\ell = 0$ and is equal to zero otherwise.
\end{conjecture}

Clearly Conjectures~\ref{chowla-conj}, \ref{hl-conj} correspond to the special cases $k=0$ and $\ell=0$ respectively of Conjecture \ref{hlc}.
One could also generalize this conjecture by replacing the forms $n + h_j, n+h'_{j'}$ by more general linear forms $a_j n + b_j, a'_{j'} n + b'_{j'}$, no two of which are scalar multiples of each other, but we do not do so here in order to simplify the notation.

Only the $k+\ell \leq 1$ cases of Conjecture~\ref{hlc} are currently known, even if one assumes the generalized Riemann hypothesis, though see~\cite{sawin-shusterman} for some recent progress in the function field case, and the recent works~\cite{lichtman}, \cite{lichtman-teravainen} for some progress on an averaged version of this conjecture.  On the other hand, it turns out (perhaps surprisingly) that some progress on this conjecture can be made under an opposing hypothesis, namely the existence of a Siegel zero.  We use the notational conventions from Heath-Brown's work~\cite{heath-brown}:

\begin{definition}[Siegel zero]  A \emph{Siegel zero} $\beta$ is a real number associated to a primitive quadratic Dirichlet character $\chi$ of conductor $q_\chi$ such that $L(\beta,\chi)=0$ and
$$ \beta = 1 - \frac{1}{\eta \log q_\chi}$$
for some $\eta \geq 10$ (which we call the \emph{quality} of the zero).
\end{definition}

The lower bound on $\eta$ is mostly in order to ensure that $\log\log \eta$ is positive; the precise numerical value of the lower bound is not important.  From Siegel's theorem we have the (ineffective) upper bound
\begin{equation}\label{eta-bound}
\eta \ll_\eps q_\chi^\eps
\end{equation}
on the quality of a Siegel zero for any $\eps>0$.

There are prior results in the literature towards Conjecture~\ref{hlc} in the presence of a Siegel zero when only either the von Mangoldt function or the Liouville function appears in the correlation. These results are due to Heath-Brown~\cite{heath-brown} in the case of two-point correlations of the von Mangoldt function, and due to Chinis~\cite{chinis} in the case of the Chowla conjecture (with previous work by Germ\'an and Kat\'ai~\cite{german-katai} on the two-point case).  We can summarize them  as follows:

\begin{theorem}[Prior results on Hardy--Littlewood--Chowla given a Siegel zero]\label{prior}  Suppose that one has a Siegel zero $\beta$ with associated conductor $q_\chi$ and quality $\eta$.
\begin{itemize}
\item[(i)] \cite[Theorem 1]{heath-brown} For any distinct fixed natural numbers $h_1,h_2$, one has
$$ \E_{n \leq x} \Lambda(n+h_1) \Lambda(n+h_2) = {\mathfrak S} + O\left( \frac{1}{\log\log \eta} \right)$$
uniformly for all $q^{250}_\chi \leq x \leq q_\chi^{300}$, where ${\mathfrak S}$ is defined by~\eqref{ss-def}.
\item[(ii)]  \cite[Theorem 2]{german-katai} One has
$$ \E_{n \leq x} \lambda(n) \lambda(n+1) \ll \frac{1}{\log\log \eta} + \epsilon(x)$$
for $q^{10}_\chi \leq x \leq q_\chi^{(\log\log \eta)/3}$, where $\epsilon(x)$ is a quantity that goes to zero as $x \to \infty$ (uniformly in the choice of Siegel zero).
\item[(iii)] \cite[Theorem 1.2]{chinis}  For any distinct fixed natural numbers $h'_1,\dots,h'_\ell$, one has
$$ \E_{n \leq x} \lambda(n+h'_1) \cdots \lambda(n+h'_\ell) \ll \frac{1}{(\log\log \eta)^{1/2} \log^{1/12} \eta}$$
for $q^{10}_\chi \leq x \leq q_\chi^{(\log\log \eta)/3}$.
\end{itemize}
\end{theorem}

The main result of this paper is the following common generalization and strengthening of these results.

\begin{theorem}[New results on Hardy--Littlewood--Chowla given a Siegel zero]\label{main} Let $0 \leq k \leq 2$ and $\ell \geq 0$, and let $h_1,\dots,h_k$, $h'_1,\dots,h'_\ell$ be fixed distinct natural numbers.  Suppose that one has a Siegel zero $\beta$ with associated conductor $q_\chi$ and quality $\eta$. Let $0 < \eps_0 < 1$ be fixed, and let $x$ lie in the range
\begin{equation}\label{x-range}
q_\chi^{10k + \frac{1}{2} + \eps_0} \leq x \leq q_\chi^{\eta^{1/2}}.
\end{equation}
Then we have
\begin{equation}\label{nx}
 \E_{n \leq x} \Lambda(n+h_1) \cdots \Lambda(n+h_k) \lambda(n+h'_1) \cdots \lambda(n+h'_{\ell}) = {\mathfrak S} + O\left(\frac{1}{\log^{\frac{1}{10 \max(1,k)}} \eta}\right)
\end{equation}
where ${\mathfrak S}$ is as in Conjecture~\ref{hlc}.  
\end{theorem}

\begin{remark}
The $k$-dependent exponent of $10k$ in the range~\eqref{x-range} can be improved somewhat, particularly when $k=1$, but we will not attempt to optimize it here. On the other hand, in order to improve the exponent $\frac{1}{2}$ in~\eqref{x-range} in the case $k=0$ it seems necessary to be able to obtain non-trivial bounds on short character sums such as
$$ \sum_{n \in I} \chi(n+h'_1) \cdots \chi(n+h'_\ell)$$
for intervals $I$ of length less than $q_\chi^{1/2}$, which is beyond the range of direct application of the Weil bounds and completion of sums (and for $\ell>1$ we were not able to adapt the Burgess argument~\cite{burgess} to such sums due to the lack of multiplicative structure).  The exponent $\frac{1}{10 \max(1,k)}$ in~\eqref{nx} can similarly be improved, but we will not attempt to do so here.
\end{remark}

Note that Theorem~\ref{main} improves the dependence on the quality $\eta$ of Siegel zero, and also allows for correlations that involve both the von Mangoldt function $\Lambda$ and the Liouville function $\lambda$, so long as the former function appears at most two times.  This latter restriction is an inherent limitation of our current state of knowledge of correlations for functions like the divisor function $\tau \coloneqq 1*1$; in particular, $k$-point correlations $\E_{n \leq x} \tau(n+h_1) \cdots \tau(n+h_k)$ are currently only well understood when $k \leq 2$.  

As a direct corollary to Theorem~\ref{main}, we can state the following strengthening of previous results. 

\begin{corollary}\label{cor}
Suppose that one has a Siegel zero $\beta$ with associated conductor $q_\chi$ and quality $\eta$. Let $0<\varepsilon_0<1$ be fixed. 
\begin{itemize}
\item[(i)]
For any distinct fixed natural numbers $h_1,h_2$, one has
$$ \E_{n \leq x} \Lambda(n+h_1) \Lambda(n+h_2) = {\mathfrak S} + O\left( \frac{1}{\log^{1/20} \eta} \right)$$
uniformly for all $q_{\chi}^{41/2+\varepsilon_0}\leq x \leq q_\chi^{\eta^{1/2}}$, where ${\mathfrak S}$ is defined by~\eqref{ss-def}.
\item[(ii)]  For any distinct fixed natural numbers $h'_1,\dots,h'_\ell$, one has
$$ \E_{n \leq x} \lambda(n+h'_1) \cdots \lambda(n+h'_\ell) \ll \frac{1}{ \log^{1/{10}} \eta}$$
uniformly for all $q^{1/2+\varepsilon_0}_\chi \leq x \leq q_\chi^{\eta^{1/2}}$.
\item[(iii)] For any fixed integer $h\neq 0$, one has
\begin{align*}
  \mathbb{E}_{|h|<p\leq x}\lambda(p+h)\ll  \frac{1}{ \log^{1/{10}} \eta}
\end{align*}
uniformly for all $q_{\chi}^{21/2+\varepsilon_0}\leq x \leq q_\chi^{\eta^{1/2}}$.
\end{itemize}
\end{corollary}

Corollary~\ref{cor}(ii) can further be applied to strengthen Chinis's result~\cite[Corollary 1.1]{chinis} on Sarnak's conjecture on M\"obius disjointness being true at infinitely many scales under the assumption of Siegel zeros. Applying Corollary~\ref{cor} and Sarnak's argument for the implication from Chowla's conjecture to Sarnak's conjecture (as in~\cite{chinis}), we see that, under the hypotheses of Corollary~\ref{cor}, for any fixed deterministic $f:\mathbb{N}\to \mathbb{C}$ we have
\begin{align*}
\sum_{n\leq x}\lambda(n)f(n)=o(x)   
\end{align*}
in the range $q_{\chi}^{1/2+\varepsilon_0}\leq x\leq q_{\chi}^{\eta^{1/2}}$. This improves on the range $q_{\chi}^{10}\leq x\leq q_{\chi}^{\log \log \eta/3}$ in~\cite{chinis}.

Corollary~\ref{cor}(iii) relates to the conjecture (considered in e.g.~\cite{pintz2}, \cite{murty-vatwani}, \cite{lichtman}, \cite{lichtman-teravainen}) that $\sum_{|h|<p\leq x}\lambda(p+h)=o(\pi(x))$, proving it for infinitely many $x$ under the existence of infinitely many Siegel zeros (of arbitrarily high quality).

We lastly note that, after the submission of this paper, Matom\"aki and Merikoski~\cite{matomaki-merikoski} proved a quantitatively stronger version of Corollary~\ref{cor}(i).

\subsection{Overview of proof}

The general strategy for proving results such as Theorem~\ref{main} is now well known: in the presence of a Siegel zero (and for $x$ comparable in log-scale to $q_\chi$), the function $\lambda$ ``pretends''\footnote{Following~\cite{pretentious}, we informally say that one arithmetic function $f$ ``pretends'' to be another $g$ if they are often close to each other when evaluated at rough numbers.} to be like the Dirichlet character $\chi$, and the von Mangoldt function $\Lambda = \mu * \log$ similarly ``pretends'' to be like $\chi * \log$, so the correlation in~\eqref{nx} is of comparable complexity to the average
$$ \E_{n \leq x} (\chi*\log)(n+h_1) \cdots (\chi*\log)(n+h_k) \chi(n+h'_1) \cdots \chi(n+h'_{\ell})$$
(in practice we also have to insert some sieve weights to account for the fact that not all numbers are rough).  This is a twisted and weighted version of the divisor correlation
$$ \E_{n \leq x} \tau(n+h_1) \cdots \tau(n+h_k)$$
which, as previously mentioned, is well understood for $k \leq 2$, basically because the Weil bounds for Kloosterman sums ensure that $\tau$ has level of distribution at least $2/3$, the key point being that this is larger than $1/2$.  The twist by $\chi$ introduces the need to estimate character sums such as
$$ \E_{n \leq x} \chi(n+h_1) \cdots \chi(n+h_k) \chi(n+h'_1) \cdots \chi(n+h'_\ell)$$
which can be adequately controlled by the Weil estimates for character sums since we are in the regime $x \gg q_\chi^{1/2}$.

To make this strategy rigorous, we will approximate the functions $\Lambda,\lambda$ by a series of  more tractable approximants that involve the exceptional character $\chi$ (as well as the scale $x$).  We will do this by executing the following steps in order:

\begin{itemize}
\item[(i)]  Replace the Liouville function $\lambda$ with an approximant $\lambda_{\Siegel}$, which is a completely multiplicative function that agrees with $\lambda$ at small primes and agrees with $\chi$ at large primes.  (This step was also performed in~\cite{german-katai}, \cite{chinis}.)
\item[(ii)]  Replace the von Mangoldt function $\Lambda$ with an approximant $\Lambda_\Siegel$, which is the Dirichlet convolution $\chi * \log$ multiplied by a Selberg sieve weight $\nu$ to essentially restrict that convolution to almost primes.  (This step essentially also appears in~\cite{heath-brown}.)
\item[(iii)]  Replace $\lambda_{\Siegel}$ with a more complicated truncation $\lambda_\Siegel^\sharp$ which has the structure of a ``Type I sum'', and which agrees with $\lambda_\Siegel$ on numbers that have a ``typical'' factorization.
\item[(iv)]  Replace the approximant $\Lambda_\Siegel$ with a more complicated approximant $\Lambda_\Siegel^\sharp$ which has the structure of a ``Type I sum''. (This step is inspired by a similar Type I approximation to the divisor function $\tau$ (and its higher order generalizations) 
recently introduced in~\cite{MSTT}, \cite{MRSTT}.)
\item[(v)]  Now that all terms in the correlation have been replaced with tractable Type I sums, use standard Euler product calculations and Fourier analysis, similar in spirit to the proof of the pseudorandomness of the Selberg sieve majorant for the primes in~\cite[Appendix D]{gt-linear}, to evaluate the correlation to high accuracy.
\end{itemize}

More succinctly, the proof of Theorem~\ref{main} proceeds by justifying all of the following approximations:
\begin{equation}\label{chain}
\begin{split}
&\E_{n \leq x} \Lambda(n+h_1) \cdots \Lambda(n+h_k) \lambda(n+h'_1) \cdots \lambda(n+h'_{\ell}) \\
&\quad \stackrel{(i)}{\approx} \E_{n \leq x} \Lambda(n+h_1) \cdots \Lambda(n+h_k) \lambda_\Siegel(n+h'_1) \cdots \lambda_\Siegel(n+h'_{\ell}) \\
&\quad \stackrel{(ii)}{\approx}\E_{n \leq x} \Lambda_\Siegel(n+h_1) \cdots \Lambda_\Siegel(n+h_k) \lambda_\Siegel(n+h'_1) \cdots \lambda_\Siegel(n+h'_{\ell}) \\
&\quad \stackrel{(iii)}{\approx}\E_{n \leq x} \Lambda_\Siegel(n+h_1) \cdots \Lambda_\Siegel(n+h_k) \lambda_\Siegel^\sharp(n+h'_1) \cdots \lambda_\Siegel^\sharp(n+h'_{\ell}) \\
&\quad \stackrel{(iv)}{\approx} \E_{n \leq x} \Lambda_\Siegel^\sharp(n+h_1) \cdots \Lambda_\Siegel^\sharp(n+h_k) \lambda_\Siegel^\sharp(n+h'_1) \cdots \lambda_\Siegel^\sharp(n+h'_{\ell}) \\
&\quad \stackrel{(v)}{\approx} {\mathfrak S}
\end{split}
\end{equation}
where the precise meaning of the symbol $\approx$ is given in~\eqref{approx-def} below.

The steps (i)-(v) are executed in Sections~\ref{sec:siegel-liouville}--\ref{sec:mainterm} respectively.  Interestingly, the hypothesis $k \leq 2$ is only used in step (iv) of this process.

Steps (i) and (ii) of the strategy  rely ultimately on the well known phenomenon that in the presence of a Siegel zero, one has $\chi(p) = -1$ for most primes $p$ that are comparable to the conductor $q_{\chi}$ in log-scale.  Traditionally, such phenomena are justified using complex-analytic methods, and in particular by exploiting the Deuring--Heilbronn phenomenon.  It turns out that an alternate approach relying almost entirely on elementary methods leads instead to significantly superior dependence on the quality $\eta$ of the zero; see Proposition~\ref{prop_Psum}. This eventually enables us to obtain a wider $x$ range in Theorem~\ref{main} than in previous results. 

Step (iii) involves splitting $\lambda_{\Siegel}$, which is a kind of character-twisted divisor sum, into two parts as $\lambda_\Siegel^\sharp+\lambda_{\Siegel}^{\flat}$, where $\lambda_\Siegel^\sharp$ accounts for the small divisors (with a smooth truncation) and $\lambda_{\Siegel}^{\flat}$ accounts for the large divisors. It turns out that $\lambda_{\Siegel}^{\flat}$ has a negligible contribution to the correlation (basically because smooth numbers become extremely rare at large scales). This is shown by first constructing a majorant for $\lambda_{\Siegel}^{\flat}$ (in Lemma~\ref{lamlam}) that after some Euler product computations is seen to be small ``on average'' in a suitable sense.\footnote{It would probably be possible to execute steps (ii) and (iii) in the opposite order, but that would offer no noteworthy simplifications, as we would still need to construct a majorant for  $\lambda_{\Siegel}^{\flat}$.}

Steps (iv) and (v) morally speaking amount to computing correlations such as
\begin{align}\label{taucorr}
\E_{n\leq x,n= a(q) }(\chi*\log(n))(\chi*\log(n+h))    
\end{align}
with power-saving error term (for $1\leq a\leq q\leq x^{\delta}$ for a small $\delta>0$), as well as correlations of the form
\begin{align}\label{nucorr}
\E_{n\leq x}f(n+h_1)\cdots f(n+h_k),    
\end{align}
where $f(n)=\sum_{d\mid n, d\leq x^{\delta}}b_d$ is a Type I sum with explicit coefficients $b_d$. However, both of these tasks are rather tedious as such; the first correlation~\eqref{taucorr} has secondary main terms of order $O(\frac{1}{\log x})$ times the main term (cf.,~\cite{est}), and we would need a fully explicit asymptotic in terms of $h,a,q$; meanwhile, evaluating the second correlation~\eqref{nucorr} with the Goldston--Y{\i}ld{\i}r{\i}m approach~\cite{goldston-yildirim} leads to some tricky contour integrals. We therefore smoothen $\Lambda_{\Siegel}$ by inserting a smooth partition of unity; the smoothness of the resulting functions makes handling error terms easier, just as in the smoothed approach to Goldston--Y{\i}ld{\i}r{\i}m type correlations in~\cite[Appendix D]{gt-linear}. We can also avoid explicitly obtaining asymptotics for sums such as~\eqref{taucorr} by using the Dirichlet hyperbola method, although the main ingredient for evaluating such correlations (namely Kloosterman sum bounds) is still needed. Our use of smooth weights does still necessitate some lengthy yet standard Fourier-analytic computations, but the arithmetic input is easier than in a direct approach involving an evaluation of~\eqref{taucorr}, \eqref{nucorr}.

\subsection{Acknowledgments}

TT was supported by a Simons Investigator
grant, the James and Carol Collins Chair, the Mathematical Analysis \& Application
Research Fund Endowment, and by NSF grant DMS-1764034. JT was supported by a
Titchmarsh Fellowship and Academy of Finland grant no. 340098. 

The authors thank Kaisa Matom\"aki and Jori Merikoski for pointing out a slight correction to the proof of Proposition~\ref{prop_Psum} in an earlier version of this paper. The authors would also like to thank the referee for helpful comments and suggestions.

\section{Notation}\label{notation-sec}

\subsection{Asymptotic notation}

For the rest of the paper, we let $k,\ell,h_1,\dots,h_k,h'_1,\dots,h'_\ell,\eps_0,\beta,\chi,q_{\chi},\eta,x$ be as in Theorem~\ref{main}, save that we will not require the hypothesis $k \leq 2$ except in Section~\ref{sec:siegel-vonmangoldt}, and that we do not impose the restriction~\eqref{x-range} on $x>1$ before Section~\ref{sec:siegel-liouville}.  We use the asymptotic notation $X \ll Y$, $Y \gg X$, or $X = O(Y)$ to denote the bound $|X| \leq CY$ where $C$ is a constant which is allowed to depend on the ``fixed'' quantities $k,\ell,h_1,\dots,h_k,h'_1,\dots,h'_{\ell},\eps_0$; we permit the constants to be ineffective.  Thus for instance the singular series ${\mathfrak S}$ in Conjecture~\ref{hlc} obeys the bound ${\mathfrak S} = O(1)$.  If we need the constant $C$ to depend on additional parameters, we will indicate this by subscripts, for instance $X \ll_A Y$ denotes the bound $|X| \leq C_A Y$ where $C_A$ depends on the parameter $A$ as well as the fixed quantities. We write $X \asymp Y$ for $X \ll Y \ll X$.  

By shrinking $\eps_0$ if necessary, we may assume that $\eps_0$ is sufficiently small depending on $k,\ell$.  We will also assume that $\eta$ is sufficiently large depending on the fixed quantities, since otherwise the claim follows from standard upper bound sieves (such as Lemma~\ref{sub}).  By~\eqref{eta-bound}, this also means that $q_{\chi}$ (and hence $x$) is also sufficiently large depending on the fixed quantities. 

\subsection{Indicator and exponential functions}

If $S$ is a sentence, we use $1_S$ to denote its indicator, thus $1_S=1$ when $S$ is true and $1_S=0$ otherwise.  If $E$ is a set, we use $1_E$ to denote the indicator function $1_E(n) \coloneqq 1_{n \in E}$.

In addition to the notation $e(\theta) \coloneqq e^{2\pi i \theta}$, we also write $e_q(a) \coloneqq e(a/q) = e^{2\pi i a/q}$ for natural numbers $q$ and $a \in \Z/q\Z$.    We also write $\|\theta\|_{\R/\Z}$ for the distance of $\theta$ to the nearest integer.



\subsection{Primes and prime factorization}

Unless otherwise specified, all sums and products will be over the natural numbers $\N = \{1,2,\dots\}$, with the exception of sums and products involving the variable $p$ (or $p'$, $p_1$, etc.), which will be over primes.  We define an \emph{exceptional prime} to be a prime $p^*$ such that $\chi(p^*) \neq -1$; sums over $p^*$ (or $p^*_1$, etc.) will always be understood to be over exceptional primes.

If $n$ is a natural number and $p$ is a prime, we let $n_{(p)}$ denote the largest power of $p$ dividing $n$, thus from the fundamental theorem of arithmetic
\begin{equation}\label{ftoa}
n = \prod_p n_{(p)}.
\end{equation}
For any threshold $z>1$, we may therefore factor a natural number $n$ as
\begin{equation}\label{nsplit}
n = n_{(\leq z)} n_{(>z)}
\end{equation}
where the $z$-smooth and $z$-rough components $n_{(\leq z)}, n_{(>z)}$ of $n$ are defined as
\begin{align*}
n_{(\leq z)} &\coloneqq \prod_{p \leq z} n_{(p)} \\
n_{(> z)} &\coloneqq \prod_{p > z} n_{(p)}.
\end{align*}
For a prime $p$, we let 
$$\N_{(p)} \coloneqq \{ n_{(p)}: n \in \N \} = \{ 1, p, p^2, \dots\}$$
denote the multiplicative semigroup generated by $p$, and similarly for a threshold $z>1$ we write
\begin{align*}
\N_{(\leq z)} &\coloneqq \{ n_{(\leq z)}: n \in \N \}\\
\N_{(> z)} &\coloneqq \{ n_{(> z)}: n \in \N \}
\end{align*}
for the multiplicative semigroups of $z$-smooth and $z$-rough numbers respectively.

If $d_1,\dots,d_m$ are natural numbers, we use $(d_1,\dots,d_m)$ and $[d_1,\dots,d_m]$ to denote their greatest common divisor and least common multiple, respectively.  We use $d\ (q)$ to denote the reduction of $d$ to $\Z/q\Z$, and $q|d$ to denote the assertion that $q$ divides $d$ (or equivalently $d = 0\ (q)$).  

A function $g: \N^m \to \C$ of $m$ natural numbers $d_1,\dots,d_m$ is \emph{multiplicative} if one has
$$ g(d_1 d'_1, \ldots, d_m d'_m) = g(d_1,\ldots,d_m) g(d'_1,\ldots,d'_m)$$
whenever $(d_1 \cdots d_m, d'_1 \cdots d'_m)=1$.  Observe the \emph{Euler product identity}
\begin{equation}\label{euler}
\sum_{d_1,\dots,d_m} g(d_1, \dots, d_m) = \prod_p E_p
\end{equation}
whenever the left-hand side is absolutely convergent, where the Euler factors $E_p$ are defined as
$$ E_p \coloneqq \sum_{d_1,\dots,d_m \in \N_{(p)}} g(d_1, \dots, d_m).$$
We observe the localized form
\begin{equation}\label{euler-local}
\sum_{d_1,\dots,d_m \in \N_{(\leq z)}} g(d_1, \dots, d_m) = \prod_{p \leq z} E_p
\end{equation}
of the Euler product identity for any threshold $z>0$; in particular, if $g$ is non-negative, then
\begin{equation}\label{euler-stop}
\sum_{d_1,\dots,d_m \leq z} g(d_1, \dots, d_m) \leq \prod_{p \leq z} E_p.
\end{equation}

We will frequently rely on Dirichlet convolution
$$ f*g(n) \coloneqq \sum_{d|n} f(d) g\left(\frac{n}{d}\right).$$
We let pointwise product take precedence over convolution, thus for instance
$$ f_1 f_2 * f_3 f_4 = (f_1 f_2) * (f_3 f_4).$$
From~\eqref{nsplit} we observe the identity
\begin{equation}\label{fmulti}
 f = f_{(\leq z)} * f_{(>z)}
\end{equation}
for any multiplicative function $f$ and any threshold $z>1$, where
\begin{align*}
f_{(\leq z)} &= f 1_{\N_{(\leq z)}} \\
f_{(>z)} &= f 1_{\N_{(>z)}} 
\end{align*}
are the restrictions of $f$ to $z$-smooth and $z$-rough numbers respectively.  Thus for instance $1_{(\leq z)} = 1_{\N_{(\leq z)}}$.  Observe that this splitting respects Dirichlet convolutions, in the sense that
\begin{equation}\label{respect}
(f*g)_{(\leq z)} = f_{(\leq z)} * g_{(\leq z)}; \quad (f*g)_{(> z)} = f_{(> z)} * g_{(> z)}
\end{equation}
for any $f,g \colon \N \to \C$.

\subsection{Scales}

We will make frequent use of the scales
\begin{equation}\label{R-def}
 R \coloneqq x^{1/\log^{\frac{1}{5\max(1,k)}} \eta}
\end{equation}
and
\begin{equation}\label{D-def}
 D \coloneqq x^{\frac{\eps_0}{10 (k+\ell)}}.
\end{equation}
We will also occasionally need the auxiliary scale
\begin{equation}\label{R0-def}
 R_0 \coloneqq x^{1/\sqrt{\log \eta}}. 
\end{equation}

The reader may wish to keep in mind the hierarchy of scales
$$ 1 < \log \eta \ll \log q_\chi \ll \log x < R_0 < R < D < x.$$
which follows easily from~\eqref{eta-bound}.  The conductor $q_\chi$ lies between $\log x$ and $x^2$ but can be either smaller or larger than $R_0$, $R$, or $D$.

We adopt the notation
$$ \log_z y \coloneqq \frac{\log y}{\log z}$$
for the logarithm of $y$ to base $z$ for any $y,z > 0$, and use the notation $X \approx Y$ as an abbreviation for
\begin{equation}\label{approx-def}
 X = Y + O\left( \frac{1}{\log^{\frac{1}{10\max(1,k)}} \eta} \right).
\end{equation}
Thus for instance the estimate~\eqref{nx} can be abbreviated to
$$ \E_{n \leq x} \Lambda(n+h_1) \cdots \Lambda(n+h_k) \lambda(n+h'_1) \cdots \lambda(n+h'_{\ell}) \approx {\mathfrak S}.$$
The scales $R_0, R, D$ have been chosen so that certain combinations of these scales with $x, \eta, q_\chi$ that will arise in our calculations are negligible with respect to the relation $\approx$.  More precisely, we observe for future reference that thanks to~\eqref{x-range}, \eqref{R-def}, \eqref{D-def}, \eqref{R0-def}, \eqref{eta-bound} we have relations
\begin{align}
\log_D x \asymp \log_x D &\asymp 1 \label{D-asymp}\\
\log_D R \asymp \log_x R = \log^{-\frac{1}{5\max(1,k)}} \eta &\approx 0 \label{lrn}\\
(\log_R^k x) \log_R R_0 = \log^{\frac{k}{5\max(1,k)}-\frac{1}{2}} \eta&\approx 0 \label{rxr0}\\
\frac{\log_R^k x}{\log^k \eta} = \log^{\frac{k}{5\max(1,k)}-k} \eta &\approx 0 \label{rxr1}\\
\frac{R^{2k} D^{2(k+\ell)} q_\chi^{4k+1/2} x^{\eps_0}}{x} \leq x^{\frac{2k}{\log^{1/(5\max(1,k))} \eta} + \frac{2 (k+\ell)\eps_0}{10(k+\ell)} + \frac{4k+1/2}{10k+\frac{1}{2}+\eps_0} + \eps_0 - 1} &\approx 0 \label{rdk0}\\
(\log_R^{O(1)} x) \exp(-\sqrt{\log \eta}/2) = (\log^{O(1)} \eta) \exp(-\sqrt{\log \eta}/2)  &\approx 0 \label{llog} \\
q_\chi^{-\frac{\eps_0}{4}} \log^{O(1)} x \leq x^{-\frac{\eps_0}{4(10k+\frac{1}{2}+\eps_0)}} \log^{O(1)} x &\approx 0\label{etao}
\end{align}
as well as the estimate
\begin{equation}\label{eta-decay}
\exp\left(- \frac{1}{8} \log_{R} D \right) = \exp\left( - \frac{\eps_0 \log^{1/5\max(1,k)} \eta}{80(k+\ell)} \right) \ll_A \log^{-A} \eta
\end{equation}
for all $A>0$.  Also, for $k=1,2$, we note for future reference that
\begin{equation}\label{dsmash}
\begin{split}
\frac{(\sqrt{x} R^2)^{\frac{3}{2}(k-1)} D^{2(k+\ell)} q_\chi^{9/2}}{x^{1-2\eps_0}}
&\ll x^{\frac{3}{4}(k-1) + \frac{9}{2} \frac{1}{10k} + 3\eps_0 - 1} \\
&\ll x^{-\eps_0} \\
&\approx 0
\end{split}
\end{equation}
since $\frac{3}{4}(k-1) + \frac{9}{2} \frac{1}{10k} \leq 1 - \frac{1}{40} < 1$ for $k=1,2$.  

\subsection{The Selberg sieve}

We fix a smooth function $\psi \colon \R \to \R$ supported on $[-1,1]$ that equals to $1$ on $[-1/2,1/2]$, and define the smooth cutoffs
\begin{equation}\label{psia}
\psi_{\leq z}(n) \coloneqq \psi(\log_z n)
\end{equation}
and
\begin{equation}\label{psibig-def}
\psi_{> z}(n) \coloneqq 1 - \psi(\log_z n)
\end{equation}
for any $z>1$.  We then define the \emph{Selberg sieve}\footnote{Here we use the Selberg sieve with smoothed coefficients, which was implicitly introduced by Goldston and Y{\i}ld{\i}r{\i}m; see for instance~\cite[Appendix D]{gt-linear} for further discussion.  Other sieve approximants to $1_{(>z)}$ could be used as a substitute for this sieve if desired; for instance the beta sieve was used in place of a Selberg sieve in the recent work \cite{matomaki-merikoski}, which appeared subsequently to the initial release of this paper.}
\begin{equation}\label{nuz-def}
\nu(n) \coloneqq \left(\sum_{d|n} \mu(d) \psi_{\leq R}(d)\right)^2.
\end{equation}
Note that $\nu$ is an upper bound sieve for $1_{(>R)}$, thus
\begin{equation}\label{nuz}
 1_{(>R)}(n) \leq \nu(n)
\end{equation}
for all natural numbers $n$.

\section{Tools}

In this section we collect some (mostly standard) estimates on various arithmetic functions which will be used in our main argument.

\subsection{Multiplicative number theory bounds}

We recall the crude \emph{divisor bound}
\begin{equation}\label{divisor-bound}
\tau(n) \ll_\eps n^\eps
\end{equation}
for any $n \geq 1$ and $\eps>0$; see e.g.,~\cite[(2.20)]{mv}.

From the Euler product formula
$$ \zeta(s) = \prod_p \left(1 - \frac{1}{p^s}\right)^{-1}$$
and the fact that $\zeta$ has a simple pole at $s=1$ with residue $1$ and no zeroes in $\{ s: |s-1| \leq \frac{1}{2}\}$, we see that
\begin{equation}\label{prod-p}
\prod_p \left(1 - \frac{1}{p^s}\right) = (1 + O(|s-1|)) (s-1)
\end{equation}
whenever $s$ is a complex number with $\Re s > 1$ and $|s-1| < \frac{1}{2}$.

From Mertens' theorem we easily verify that
\begin{equation}\label{mertens}
 \sum_{p \leq z} \frac{\min(\sigma \log_R p, 1)}{p} \ll \log(1 + \sigma \log_R z) 
\end{equation}
for any $\sigma > 0$ and $R,z \geq 1$, as can be seen by verifying the cases $\sigma \log_R z < 1$ and $\sigma \log_R z \geq 1$ separately; in exponential form we thus have
\begin{equation}\label{mertens-alt}
 \prod_{p \leq z}\left(1 + O\left(\frac{\min(\sigma \log_R p, 1)}{p}\right)\right) \leq (1 + \sigma \log_R z)^{O(1)}.
\end{equation} 
Mertens' theorem also gives (by dyadic decomposition) the bounds
\begin{equation}\label{mertens-2}
\sum_{p \geq y} \frac{1}{p^{1+1/\log z}} = \sum_{p \geq y} \frac{\exp(-\log_z p)}{p} \asymp \frac{\exp( - \log_z y )}{\log_z y}
\end{equation}
and
\begin{equation}\label{mertens-3}
\prod_{p \leq z} (1 - \frac{m}{p}) \asymp_m \log^{-m} z
\end{equation}
for any $y \geq z \geq 2$ and $m \geq 1$.  In particular 
\begin{equation}\label{mertens-4}
\prod_{p} \left(1 + O\left(\frac{1}{p^{1+1/\log z}}\right)\right) \ll \log^{O(1)} z.
\end{equation}

We recall an elementary inequality of Landreau~\cite{landreau} that allows one to upper bound the divisor function $\tau$ by a Type I sum:

\begin{lemma}[Landreau's inequality]\label{landreau-lemma}\ 
\begin{itemize}
\item[(i)]  If $n$ is a natural number and $y > z > 1$, then we can factor
\begin{equation}\label{nfact}
 n = n_{(>z)} n_1 \cdots n_m
\end{equation}
where $n_1,\dots,n_m \leq y$ lie in $\N_{(\leq z)}$ and $0 \leq m \leq 1 + \log_{y/z} n$.  Also, $n_{(>z)}$ is the product of at most $\log_z n$ primes.
\item[(ii)] If $\eps>0$, then
\begin{equation}\label{landreau-special-1}
\tau(n) \ll_\eps \sum_{d|n: d \leq n^\eps} \tau(d)^{O_\eps(1)}
\end{equation}
for all $n \geq 1$. In particular, by~\eqref{D-def}, one has
\begin{equation}\label{landreau-special}
\tau(n) \ll \sum_{d|n: d \leq D} \tau(d)^{O(1)}
\end{equation}
for $n \ll x$.
\end{itemize}
\end{lemma}

\begin{proof}  Observe from the greedy algorithm that any number in $\N_{(\leq z)}$ is either greater than $y$, or contains a factor between $y/z$ and $y$.  Iterating this fact, we can factor $n_{(\leq z)} = n_1 \cdots n_m$ where $n_1,\dots,n_m \leq y$ and all but at most one of the $n_1,\dots,n_m$ are greater than or equal to $y/z$.  This gives the bound $m \leq 1 + \log_{y/z} n$.  Since $n_{(>z)}$ is the product of primes greater than $z$, the total number of primes is at most $\log_z n$.  This gives (i).

For (ii), we apply (i) with $y = n^\eps$ and $z = n^{\eps/2}$ and use~\eqref{D-asymp} to obtain the factorization~\eqref{nfact} with $n_{(>n^{\eps/2})}$ the product of $O_\eps(1)$ primes, $m= O_\eps(1)$, and $n_1,\dots,n_m \leq n^\eps$.  Using the elementary inequality $\tau(ab) \leq \tau(a) \tau(b)$, we conclude that
$$ \tau(n) \ll \tau(n_1) \cdots \tau(n_m)$$
and hence by the pigeonhole principle
$$ \tau(n) \ll \tau(d)^m$$
for $d$ equal to one of the $n_1,\dots,n_m$.  The claim (ii) follows.
\end{proof}
 
We also record a standard sieve upper bound, which can easily be deduced from the fundamental lemma of sieve theory (or the large sieve):

\begin{lemma}[Sieve upper bound]\label{sub} Suppose that for every prime $p \leq x$ there is a natural number $0 \leq \omega(p) \ll 1$, and let $E$ be a subset of $\{n: n \leq x\}$ which avoids at least $\omega(p)$ residue classes modulo $p$ for each $p \leq x$.  Then we have
$$ \E_{n \leq x} 1_E(n) \ll \prod_{p \leq x} \left(1 - \frac{\omega(p)}{p}\right).$$
\end{lemma}

\begin{proof} We may assume that $\omega(p)<p$ for all $p$, since otherwise $E$ is empty and the claim is trivial (of course, this assumption is only non-trivial for the very small primes $p = O(1)$).  
By Mertens' theorem the contribution of the primes $x^{1/100} < p \leq x$ to the right-hand is $\asymp 1$, so we may replace the product $\prod_{p \leq x}$ here with $\prod_{p \leq x^{1/100}}$.

Let $g$ be the multiplicative function $g(d) \coloneqq \prod_{p|d: p \leq x^{1/100}} \frac{\omega(p)}{p}$.  
By the fundamental lemma of sieve theory (see~\cite[Lemma 6.3]{ik}) we can find weights $\lambda^+_d \in [-1,1]$ for all $d \leq \sqrt{x}$ such that
\begin{equation}\label{1n}
 1_{n=1} \leq \sum_{d|n} \lambda^+_d
\end{equation}
for all natural numbers $n$, and
\begin{equation}\label{sun}
 \sum_{d \leq x^{1/2}} \lambda^+_d g(d) \ll \prod_{p \leq x^{1/100}} \left( 1 - g(p)\right).
\end{equation}
For each $d \leq x^{1/2}$, let $E_d$ be the set formed by removing the $\omega(p)$ residue classes modulo $p$ from $\{n: n \leq x\}$ for all $p|d$. Then from~\eqref{1n} (with $n$ replaced by $\prod_{p \leq x^{1/100}: n \not \in E_p} p$) we have the pointwise bound
$$ 1_E \leq \sum_{d \leq x^{1/2}} \lambda^+_d 1_{E_d}$$
and thus
$$ \E_{n \leq x} 1_E(n) \leq \sum_{d \leq x^{1/2}} \lambda^+_d \E_{n \leq x} 1_{E_d}(n).$$
From the Chinese remainder theorem we have
$$ \E_{n \leq x} 1_{E_d}(n) = g(d) + O\left(\frac{1}{x}\right)$$
and thus by~\eqref{sun}
$$ \E_{n \leq x} 1_{E_d}(n) \ll \prod_{p \leq x^{1/100}}\left(1 - \frac{\omega(p)}{p}\right) + \frac{1}{x^{1/2}}.$$
By Mertens' theorem~\eqref{mertens-3}, the second term on the right-hand side is certainly dominated by the former, and the claim follows.
\end{proof}

We also record the following easy consequence of the Chinese remainder theorem.

\begin{lemma}[Chinese remainder theorem]\label{chin}  Let $d_1,\dots,d_{k'}$ be natural numbers for some $2 \leq k' \leq k+\ell$, and set $d \coloneqq [d_1,\dots,d_{k'}]$.  
\begin{itemize}
\item[(i)] If $(d_i, d_j)$ does not divide $h_i-h_j$ for some $1 \leq i < j \leq k'$, then $\prod_{j=1}^{k'} 1_{d_j|n+h_j}$ vanishes for all $n$.
\item[(ii)] If instead $(d_i, d_j)$ divides $h_i-h_j$ for all $1 \leq i < j \leq k'$, then there is a unique residue class $a\ (d)$ such that
$\prod_{j=1}^{k'} 1_{d_j|n+h_j} = 1_{n = a\ (d)}$.  Furthermore, $a = - h_j\ (d_j)$ for all $j=1,\dots,k'$, and $d \asymp d_1 \cdots d_{k'}$.
\end{itemize}
\end{lemma}

\begin{proof} All the claims are immediate except for the existence of the residue class $a$ in part (ii) (the final part of (ii) following from the general relation $\frac{d_1 \cdots d_{k'}}{[d_1,\dots,d_{k'}]} | \prod_{1 \leq i < j \leq k'} (d_i,d_j)$).  By the Chinese remainder theorem we may assume that the $d_i$ are all powers of a single prime $p$.  Then we have $d = d_i$ for some $1 \leq i \leq k'$, and the claim follows by setting $a \coloneqq -h_i$.
\end{proof}

\subsection{Some Fourier analysis}

Recall the \emph{Fourier inversion formula}: if $g \colon \R \to \C$ is a Schwartz function, then one has
\begin{equation}\label{fourier-inverse}
 g(u) = \int_\R e^{-itu} f(t)\ dt
\end{equation}
for all $u \in \R$, where the Fourier transform $f \colon \R \to \C$ of $g$ is another Schwartz function defined by the formula
$$ f(t) \coloneqq \frac{1}{2\pi} \int_\R e^{itu} g(u)\ du.$$
As a special case of this, if $\varphi \colon \R \to \C$ is a function such that $u \mapsto e^u \varphi(u)$ is Schwartz, then
$$ e^u \varphi(u) = \int_\R e^{-itu} f(t)\ dt$$
for all $u \in \R$, where
$$ f(t) \coloneqq \frac{1}{2\pi} \int_\R e^{(1+it)u} \varphi(u)\ du.$$
In particular, for any real $n,z > 0$ we have
\begin{equation}\label{fourier}
 \varphi( \log_z n ) = \int_\R \frac{1}{n^{\frac{1+it}{\log z}}} f(t)\ dt.
\end{equation}
Evaluating this formula at $n=1$ we conclude that
\begin{equation}\label{fourier-0}
\varphi(0) = \int_\R f(t)\ dt
\end{equation}
and if one differentiates at $n=1$ instead one obtains the variant identity
\begin{equation}\label{fourier-1}
-\varphi'(0) = \int_\R (1+it) f(t)\ dt.
\end{equation}

As an application of these Fourier representations, we give an analogue of Lemma~\ref{sub} for the Selberg sieve $\nu$ (cf.,~\cite[Lemma 3]{pintz},~\cite[Proposition 14]{polymath}):

\begin{lemma}[Selberg sieve concentrates on almost primes]\label{nusieve}  Let $0 \leq \ell' \leq \ell$, and let $1 \leq d_1,\dots,d_k,d'_1,\dots,d'_{\ell'} \leq x$ be integers. Then
\begin{equation}\label{enix}
\begin{split}
&\E_{n \leq x} \prod_{j=1}^k \nu(n+h_j) 1_{d_j|n+h_j} \prod_{j'=1}^{\ell'} 1_{d'_{j'}|n+h'_{j'}} \\
&\quad \ll_{A} \frac{\tau(d_1 \cdots d_k d'_1 \cdots d'_{\ell'})^{O(1)}}{d_1 \cdots d_k d'_1 \cdots d'_{\ell'} \log^k R} \int_1^\infty \left(\prod_{ p|d_1 \cdots d_k} \min(\sigma \log_R p,1)\right) \frac{d\sigma}{\sigma^A} + \frac{R^{2k}}{x}
\end{split}
\end{equation}
for any $A>0$.
\end{lemma}

The $\frac{R^{2k}}{x}$ error term is negligible in practice.  The $\sigma$ variable of integration is technical and as a first approximation the reader is invited to replace $\sigma$ with $1$ (and delete the integral).  The key feature of this estimate are the factors of $\min(\sigma \log_R p,1)$, which make the left-hand side of~\eqref{enix} small when $d_1,\dots,d_k$ have one or more small prime factors.  With further effort one could obtain a more precise asymptotic for the left-hand side of~\eqref{enix} (in the spirit of~\cite[Theorem D.3]{gt-linear}) but we will not need to do so here.

\begin{proof}  It is convenient to relabel by writing $k' \coloneqq k + \ell'$ and $h_{k+j} \coloneqq h'_j$, $d_{k+j} \coloneqq d'_j$ for $j=1,\dots,\ell'$.  By Lemma~\ref{chin} we may assume that $(d_i,d_j)|h_i-h_j$ for all $1 \leq i < j \leq k'$.  In particular, if we set $d \coloneqq [d_1,\dots,d_k]$ and $d' \coloneqq [d_1,\dots,d_{k'}]$, then $d \asymp d_1 \cdots d_k$ and $d' \asymp d_1 \cdots d_{k'}$.

By~\eqref{nuz-def}, the left-hand side of~\eqref{enix} may be expanded as
$$ \sum_{d'_1,d''_1,\dots,d'_k,d''_k} \E_{n \leq x} \prod_{j=1}^k \mu\psi_{\leq R}(d'_j) \mu\psi_{\leq R}(d''_j) \prod_{j=1}^{k'} 1_{d^*_j|n+h_j}$$
where $d_j^* \coloneqq [d_j,d'_j,d''_j]$ for $j=1,\dots,k$ and $d_j^* \coloneqq d_j$ for $j=k+1,\dots,k'$.  From Lemma~\ref{chin} we see that the average $\E_{n \leq x} \prod_{j=1}^{k'} 1_{d^*_j|n+h_j}$ vanishes unless $(d^*_i,d^*_j)|h_i-h_j$ for all $1 \leq i < j \leq k'$, in which case it is equal to $\frac{1}{[d^*_1,\dots,d^*_{k'}]} + O(\frac{1}{x})$.  The contribution of the error $O(\frac{1}{x})$ is of size $O(\frac{R^{2k}}{x})$, so it suffices to show that
\begin{align*} &\sum_{d'_1,d''_1,\dots,d'_k,d''_k} \frac{\prod_{j=1}^k \mu\psi_{\leq R}(d'_j) \mu\psi_{\leq R}(d'_j) \prod_{1 \leq i < j \leq k'} 1_{(d^*_i, d^*_j)|h_i-h_j}}{[d^*_1,\dots,d^*_{k'}]}\\
&\ll_A \int_1^\infty \frac{\tau(d'')^{O(1)}}{d' \log^k R} \prod_{p|d} \min(\sigma \log_R p,1)\ \frac{d\sigma}{\sigma^A}.
\end{align*}
We can expand the left-hand side using~\eqref{fourier} and Fubini's theorem as
$$\int_{\R^{2k}}
 \sum_{d'_1,d''_1,\dots,d'_k,d''_k} \frac{\prod_{j=1}^k \mu(d'_j) \mu(d''_j) \prod_{1 \leq i < j \leq k'} 1_{(d^*_i, d^*_j)|h_i-h_j}}{[d^*_1,\dots,d^*_{k'}] \prod_{j=1}^k (d'_j)^{\frac{1+it'_j}{\log R}}(d''_j)^{\frac{1+it''_j}{\log R}}}\ \prod_{j=1}^k f(t'_j) f(t''_j)\, dt'_j dt''_j$$
for some Schwartz function $f$.
Changing variables using the substitution $\sigma \coloneqq 1 + \sum_{j=1}^k |t'_j| + |t''_j|$, and using the rapid decay of $f$ and the triangle inequality, it will suffice to establish the pointwise bound
$$
 \sum_{d'_1,d''_1,\dots,d'_k,d''_k} \frac{\prod_{j=1}^k \mu(d'_j) \mu(d''_j)  \prod_{1 \leq i < j \leq k'} 1_{(d^*_i, d^*_j)|h_i-h_j}}{[d^*_1,\dots,d^*_{k'}]\prod_{j=1}^k (d'_j)^{\frac{1+it'_j}{\log R}}(d''_j)^{\frac{1+it''_j}{\log R}} } \ll
\sigma^{O(1)} \frac{\tau(d')^{O(1)}}{d' \log^k R} \prod_{p|d} \min(\sigma \log_R p,1)$$
for all $t'_1,\dots,t''_k \in \R$.

By~\eqref{euler} and the fact that $\mu$ is supported on square-free numbers, the left-hand side factors as an Euler product $\prod_p E_p$ where
$$ E_p \coloneqq F_p\left(\frac{1+it'_1}{\log R},\frac{1+it''_1}{\log R},\dots,\frac{1+it'_k}{\log R},\frac{1+it''_k}{\log R}\right)$$
and
$$ F_p(z'_1,z''_1,\dots,z'_k,z''_k) \coloneqq
 \sum_{d'_1,d''_1,\dots,d'_k,d''_k \in \{1,p\}} \frac{\prod_{j=1}^k \mu(d'_j) \mu(d''_j)  \prod_{1 \leq i < j \leq k'} 1_{((d^*_i)_{(p)}, (d^*_j)_{(p)})|h_i-h_j}}{[d^*_1,\dots,d^*_{k'}]_{(p)} \prod_{j=1}^k (d'_j)^{z'_j} (d''_j)^{z''_j}}.$$
From the triangle inequality we have  
\begin{equation}\label{ept-alt}
 E_p = 1 + O\left( \frac{1}{p^{1+\frac{1}{\log R}}}\right)
\end{equation}
when $d'_{(p)}=1$ and
\begin{equation}\label{epd}
 E_p \ll \frac{1}{d'_{(p)}} 
\end{equation}
when $d'_{(p)}>1$, hence by~\eqref{mertens-2} we have
$$ \prod_{p>R} E_p d'_{(p)} \ll \tau(d')^{O(1)}.$$
Now let $p \leq R$ and $d'_{(p)} > 1$.  Then from the triangle inequality we have
$$ F_p(z'_1,z''_1,\dots,z'_k,z''_k) \ll \frac{1}{d'_{(p)}}$$
whenever $z'_1,z''_1,\dots,z'_k,z''_k$ are complex numbers of size $O(\frac{1}{\log p})$, while from the cancellation in the M\"obius coefficients $\mu(d'_j), \mu(d''_j)$ and the hypothesis $(d_i,d_j)|h_i-h_j$ for all $1 \leq i < j \leq k'$ we see that
$$ F_p(0,\dots,0) = 0.$$
From the Cauchy integral formula (in the case $\sigma \log_R p \leq 1$) or~\eqref{epd} (otherwise) we have
$$ E_p d'_{(p)} \ll \min( \sigma \log_R p, 1).$$
Finally, suppose that $p \leq R$ and $d'_{(p)}=1$.  Then from the triangle inequality we have
$$ F_p(z'_1,z''_1,\dots,z'_k,z''_k) = 1 + O\left(\frac{1}{p}\right)$$
whenever $z'_1,z''_1,\dots,z'_k,z''_k$ are complex numbers of size $O(\frac{1}{\log p})$, while from noting that the conditions
$((d^*_i)_{(p)}, (d^*_j)_{(p)})|h_i-h_j$ permit $d'_j, d''_j$ to equal $p$ for $h_j$ in at most one residue class $a\ (p)$, we have
\begin{align*}
 F_p(0,\dots,0) &= 1 + \sum_{a \in \Z/p\Z} \frac{\sum_{d'_j, d''_j \in \{1,p\} \hbox{ when } h_j = a\ (p)} \prod_{h_j=a\ (p)} \mu(d'_j) \mu(d''_j) - 1}{p} \\
&= 1 + \sum_{a \in \Z/p\Z} \frac{ (1 - 1)^{2 \# \{ j=1,\dots,k: h_j = a\ (p)\}} - 1}{p} \\
&= 1 - \frac{\# \{ h_j\ (p): j=1,\dots,k\}}{p} \\
&= \beta_p \left(1-\frac{1}{p}\right)^k
\end{align*}
thanks to~\eqref{betap-def}.  From the Cauchy integral formula (in the case $\sigma \log p \leq \log R$) or~\eqref{ept-alt} (otherwise) we thus have
$$ E_p = E_p d'_{(p)} = \beta_p \left(1-\frac{1}{p}\right)^k + O\left( \frac{\min( \sigma \log_R p, 1)}{p} \right),$$
and hence by~\eqref{betap-asym}
$$ E_p d'_{(p)} \ll \left(1-\frac{k}{p}\right) \left( 1 + O\left( \frac{\min( \sigma \log_R p, 1)}{p} \right) + O\left(\frac{1}{p^2}\right) \right).$$
From Mertens' theorem~\eqref{mertens-alt} we have
$$ \prod_{p \leq R} \left(1 + O\left(\frac{\min( \sigma \log_R p, 1)}{p}\right) + O\left(\frac{1}{p^2}\right) \right) \ll \sigma^{O(1)}.$$
Putting all this together, we see that
$$ \prod_p E_p d'_{(p)} \ll \sigma^{O(1)} \tau(d')^{O(1)} \prod_{p \leq R: p \nmid d'} \left(1-\frac{k}{p}\right)  \prod_{p \leq R: p|d} \min(\sigma \log_R p, 1)$$
and hence by Mertens' theorem~\eqref{mertens-3}
$$ \prod_p E_p d'_{(p)} \ll \frac{\sigma^{O(1)} \tau(d'')^{O(1)}}{\log^k R}  \prod_{p \leq R: p|d} \min(\sigma \log_R p, 1).$$
From~\eqref{ftoa} we have
$$\prod_p d''{(p)} = d' \asymp d_1 \cdots d_{k'},$$
and the claim follows.
\end{proof}

\subsection{Elementary consequences of a Siegel zero}\label{sec:consequence}

Recall from Section~\ref{notation-sec} that we use $p^*$ to denote primes that are \emph{exceptional} in the sense that $\chi(p^*) \neq -1$.  It is a well known phenomenon that exceptional primes become rare at scales comparable in log-scale to $q_{\chi}$. For instance, in~\cite[Lemma 3]{heath-brown} it was shown that\footnote{Strictly speaking, these results only claim to control the set where $\chi(p^*) = 1$, ignoring the relatively small number of primes where $\chi(p^*)=0$, but it is not difficult to modify the arguments to also include the latter set.}
\begin{equation}\label{hb-est}
\sum_{p^* \leq q_\chi^{500}} \frac{\log p^*}{p^*} \ll \frac{\log q_\chi}{\sqrt{\log \eta}}
\end{equation}
while in~\cite{german-katai} it was shown more generally that
\begin{equation}\label{gk-est}
\sum_{p^* \leq x} \frac{\log p^*}{p^*} \ll \exp\left( \log_{q_\chi} x \right) \frac{\log q_\chi}{\sqrt{\log \eta}}
\end{equation}
for $q_\chi^{10} \leq x \leq q_\chi^{\log\log \eta / 3}$.  In fact we can do better:

\begin{proposition}\label{prop_Psum}  Let $\eps>0$. Then
for any $x \geq q_\chi^{\frac{1+\eps}{2}}$, one has
\begin{equation}\label{excep-1}
\sum_{q_\chi^{\frac{1+\eps}{2}} < p^* \leq x} \frac{1}{p^*} \ll_\eps \frac{\log_{q_\chi} x
}{\eta} \end{equation}
and for any natural number $m \geq 2$, we have
\begin{equation}\label{excep-2}
\sum_{q_\chi^{\frac{1+\eps}{2m}} < p^* \leq q_\chi^{\frac{1+\eps}{2(m-1)}}} \frac{1}{p^*} \ll_\eps \frac{m}{\eta^{1/m}}.
\end{equation}
\end{proposition} 

The first bound is non-trivial for $x$ as large as $q_\chi^{\eta^{1-\eps_0}}$, while the second bound is non-trivial for primes $p^*$ as small as $q_\chi^{1/\log^{1-\eps_0} \eta}$. It is not difficult to recover~\eqref{gk-est} (and hence~\eqref{hb-est}) from the above proposition by taking a suitable linear combination of~\eqref{excep-1} and~\eqref{excep-2} for $m \leq \sqrt{\log \eta}$, and using Mertens' theorem to control the contribution of exceptional primes $p^* \leq q_\chi^{10/\sqrt{\log \eta}}$ (say); we leave the details to the interested reader.

\begin{proof}  For any $x \geq q_\chi^{\frac{1+\eps}{2}}$ we have from~\cite[Exercise 11.2.3(g)]{mv} that
$$\sum_{n \leq x} \frac{1*\chi(n)}{n} = (\log x + \gamma) L(1,\chi) + L'(1,\chi) + O_\eps(q_\chi^{-\eps/10}).$$
From Siegel's theorem we have $L(1,\chi) \gg_\eps q_\chi^{-\eps/10}$, and hence
\begin{equation}\label{nxb}
\sum_{n \leq x} \frac{1*\chi(n)}{n} = L(1,\chi) \left(\log x + \frac{L'}{L}(1,\chi) + O_\eps(1) \right).
\end{equation}
From~\cite[Theorem 11.4]{mv} we also have $\frac{L'}{L}(1,\chi) \asymp \eta \log q_\chi$. Thus, \eqref{nxb} gives
\begin{align}\label{L1chi}
\sum_{n \leq q_\chi^{\frac{1+\eps}{2}}} \frac{1*\chi(n)}{n} \gg L(1,\chi) \eta \log q_\chi,
\end{align}
while applying~\eqref{nxb} with $x$ replaced by $q_\chi^{\frac{1+\eps}{2} }, x q_\chi^{\frac{1+\eps}{2} }$ and subtracting we obtain 
\begin{equation}\label{nxt}
\sum_{q_\chi^{\frac{1+\eps}{2} } < n \leq x q_\chi^{\frac{1+\eps}{2} } } \frac{1*\chi(n)}{n} =L(1,\chi) (\log x+O_\eps(1)).
\end{equation}

On the other hand, from the non-negativity and multiplicativity of $1*\chi$ we have
$$\sum_{q_\chi^{\frac{1+\eps}{2} }< n \leq x q_\chi^{\frac{1+\eps}{2} } } \frac{1*\chi(n)}{n} \geq \left(\sum_{n \leq q_\chi^{\frac{1+\eps}{2}}} \frac{1*\chi(n)}{n}\right) \left(\sum_{q_\chi^{\frac{1+\eps}{2}} < p \leq x} \frac{1*\chi(p)}{p}\right).$$

Since $1*\chi(p)$ is non-negative and is at least one when $p$ is exceptional, the first claim~\eqref{excep-1} follows.

In a similar vein,  since any $n\leq q_{\chi}^{10}$ has $\leq \binom{20m}{m}$ representations in the form $n'p_1\cdots p_m$ with $q_{\chi}^{(1+\varepsilon)/2}<p_1<p_2<\cdots <p_m$, we have for any natural number $m \geq 2$ that
\begin{align*}
&\sum_{q_\chi^{\frac{1+\eps}{2} } < n \leq q_\chi^{10} } \frac{1*\chi(n)}{n}\\ 
&\geq \binom{20m}{m}^{-1}\sum_{n < q_\chi^{\frac{1+\eps}{2}}} \frac{1*\chi(n)}{n} \sum_{q_\chi^{\frac{1+\eps}{2m}} < p_1 < \cdots < p_m \leq q_\chi^{\frac{1+\eps}{2(m-1)}}: p_1,\dots,p_m \nmid n} \frac{1*\chi(p_1)}{p_1} \cdots \frac{1*\chi(p_m)}{p_m} \\
&\geq \binom{20m}{m}^{-1}\sum_{n < q_\chi^{\frac{1+\eps}{2}}} \frac{1*\chi(n)}{n} \sum_{q_\chi^{\frac{1+\eps}{2m}} < p^*_1 < \cdots < p^*_m \leq q_\chi^{\frac{1+\eps}{2(m-1)}}: p^*_1,\cdots,p^*_m \nmid n} \frac{1}{p_1^* \cdots p_m^*} \\
&= \binom{20m}{m}^{-1}\sum_{n < q_\chi^{\frac{1+\eps}{2}}} \frac{1*\chi(n)}{n}\ \frac{\sum_{q_\chi^{\frac{1+\eps}{2m}} < p^*_1, \dots, p^*_m \leq q_\chi^{\frac{1+\eps}{2(m-1)}}: p^*_1,\dots,p^*_m \nmid n, \hbox{ distinct}} \frac{1}{p_1^* \cdots p_m^*}}{m!}.
\end{align*}
 Observe that once $n<q_{\chi}^{(1+\varepsilon)/2}$ and some of the exceptional primes $p^*_1,\dots,p^*_j$, $j < m$ have been chosen, the restrictions that the exceptional prime $p^*_{j+1}$ be distinct from $p^*_1,\dots,p^*_j$ and not divide $n$ only excludes at most $2m$ primes $p^*_{j+1}$ from the range $q_\chi^{\frac{1+\eps}{2m}} < p^*_{j+1} \leq
q_\chi^{\frac{1+\eps}{2(m-1)}}$, since $n$ has at most $m$ factors in this range.  Thus we have
$$
\sum_{q_\chi^{\frac{1+\eps}{2} } < n \leq q_\chi^{10} } \frac{1*\chi(n)}{n} 
\geq \binom{20m}{m}^{-1}\left(\sum_{n < q_\chi^{\frac{1+\eps}{2}}} \frac{1*\chi(n)}{n}\right) \frac{\left(\asum_{q_\chi^{\frac{1+\eps}{2m}} < p^* \leq q_\chi^{\frac{1+\eps}{2(m-1)}} } \frac{1}{p^*}\right)^m}{m!},
$$
where the asterisk in the sum means that we are allowed to delete the $2m$ largest terms from the sum (or delete the sum entirely, i.e. replace it by zero, if there are fewer than $2m$ terms in all). The estimates~\eqref{L1chi},~\eqref{nxt} then give
$$\asum_{q_\chi^{\frac{1+\eps}{2m}} < p^* \leq q_\chi^{\frac{1+\eps}{2(m-1)}}} \frac{1}{p^*} \ll_\eps \left( \frac{\binom{20m}{m}m! L(1,\chi) \log q_\chi}{L(1,\chi) \eta \log q_\chi}\right)^{1/m} \ll \frac{m}{\eta^{1/m}}.$$
One can reinstate the top $2m$ terms from the sum on the left-hand side, since their contribution is $\ll m/q_\chi^{1/(2m)}\ll m\eta^{-1/m}$ by the Siegel bound~\eqref{eta-bound}.  The claim~\eqref{excep-2} follows.
\end{proof}

\begin{corollary}\label{excep-bound}  Let $q_{\chi}^{(1+\varepsilon)/2}\leq x\leq q_{\chi}^{\eta^{1/2}}$. We have
$$
\sum_{R_0 \leq p^* \leq x} \frac{1}{p^*} \ll \exp( - \sqrt{\log \eta} / 2)$$
and
$$
\sum_{p^*} \frac{\min(\log^{0.1} \eta \log_R p^*,1)}{(p^*)^{1+\frac{1}{\log x}}} \ll \frac{1}{\log^{0.3} \eta}.$$
\end{corollary}

This bound will be used in steps (i), (ii) of the main argument.

\begin{proof} From~\eqref{excep-1} (with $\eps=1$) and~\eqref{x-range} we have
$$ \sum_{q_\chi < p^* \leq x} \frac{1}{p^*} \ll \frac{\sqrt{\eta}}{\eta} \ll \exp( - \sqrt{\log \eta} / 2)$$
and from~\eqref{excep-2} we similarly have
$$
\sum_{q_\chi^{\frac{1}{m}} < p^* \leq q_\chi^{\frac{1}{m-1}}} \frac{1}{p^*} \ll \frac{m}{\eta^{1/m}}$$
for all $m \geq 2$.  Summing over $2 \leq m \leq \sqrt{\log \eta} + 1$, we obtain the first claim.

Now we prove the second claim.  The contribution of those $p^*$ with $p^* \geq x^{\log^{0.1} \eta}$ is acceptable by~\eqref{mertens-2}, while the contribution of those $p^*$ with $p^* \leq R^{1/\log^{0.4} \eta}$ is also acceptable by ~\eqref{mertens}.  Thus it remains to show that
$$ \sum_{R^{1/\log^{0.4} \eta} < p^* < x^{\log^{0.1} \eta}} \frac{1}{p^*} \ll \frac{1}{\log^{0.3} \eta}.$$
The contribution of those $p^*$ with $q_\chi < p^* < x^{\log^{0.1} \eta}$ is acceptable by~\eqref{excep-1} (for $\eps=1$), \eqref{x-range}, while the contribution of those $p^*$ with $R^{1/\log^{0.4} \eta }< p^* \leq q_\chi$ is acceptable by~\eqref{excep-2} (for $\eps=1$ and $2 \leq m \ll \log^{0.5} \eta$, say) and~\eqref{R-def}.
\end{proof}

\subsection{Consequences of the Weil bound for character sums}

Let $f \colon \Z \to \Z$ be a polynomial of degree $O(1)$.  If $p$ is a prime we have the standard Weil bounds
$$ \sum_{n \in \Z/p\Z} \chi_p( f(n)) e_p(an) \ll p^{1/2}$$
uniformly for all integers $a$ whenever $f$ is not a constant multiple of perfect square modulo $p$, where $\chi_p$ is the quadratic character modulo $p$; see~\cite{weil} (or~\cite{perelmuter}).  When $f$ is a constant multiple of a perfect square, we can of course use the trivial bound of $O(p)$. Since the exceptional modulus $q_\chi$ is a fundamental discriminant, it is of the form $2^j p_1 \cdots p_m$ for some $j \leq 3$ and distinct odd primes $p_1,\dots,p_m$, and so from the Chinese remainder theorem we obtain the bounds
$$ \sum_{n \in \Z/q_\chi\Z} \chi( f(n)) e_{q_\chi}(an) \ll \tau(q_\chi)^{O(1)} q_\chi^{1/2} d^{1/2}$$
uniformly in $a$, where $d$ is the largest factor of $q_\chi$ for which $f$ is a constant multiple of a perfect square modulo $d$.  Applying~\eqref{divisor-bound} and completion of sums (see~\cite[Lemma 12.1]{ik}), we conclude that
$$ \sum_{n \in I} \chi( f(n)) \ll_\eps q_\chi^{1/2+\eps} d^{1/2}$$
for any interval $I$ of length at most $q_\chi$ and any $\eps>0$; by subdividing longer intervals into intervals of length $q_\chi$, plus a remainder, we conclude that
\begin{equation}\label{sumni}
 \sum_{n \in I} \chi( f(n)) \ll_\eps q_\chi^{1/2+\eps} d^{1/2} \left( \frac{|I|}{q_\chi} + 1 \right)
\end{equation}
for any interval $I$ and any $\eps>0$.

This gives us the following bounds:

\begin{lemma}\label{and}  Let $d_1,\dots,d_{k+\ell}$ be natural numbers.  Let $I$ be an interval in $[1,x]$.
Let $J$ be a non-empty subset of $\{1,\dots,k+\ell\}$, and for each $j \in J$, let $d'_j$ be a factor of $d_j$. Then
$$ \E_{n \leq x} 1_I(n) \left(\prod_{j=1}^{k+\ell} 1_{d_j | n+h_j}\right) \prod_{j \in J} \chi\left(\frac{n+h_j}{d'_j}\right)
\ll_\eps q_\chi^{1/2+\eps} (d_1 \cdots d_{k+\ell},q_\chi)^{1/2} \left( \frac{1}{q_\chi d_1 \cdots d_{k+\ell}} + \frac{1}{x} \right)$$
for any $\eps>0$, where we use the notation $h_{k+j} \coloneqq h'_j$ for $j=1,\dots,\ell$.
\end{lemma}

This bound will be used in step (v) of the main argument, to dispose of any ``Type I sum'' contributions that are twisted by one or more factors of the exceptional character $\chi$.

\begin{proof} 
By Lemma~\ref{chin} we may assume that $(d_i,d_j)|h_i-h_j$ for all $1 \leq i < j \leq k+\ell$ and replace the conditions $d_j | n+h_j$ with $n = a\ (d)$ where
$$ d \coloneqq [d_1,\dots,d_{k+\ell}] \asymp d_1 \cdots d_{k+\ell}$$
and
$$a = -h_j\ (d_j)$$
for $j=1,\dots,k+\ell$.  Our task is now equivalent to showing that
$$ \sum_{n: dn+a \in I} \prod_{j \in J} \chi\left(\frac{dn+a+h_j}{d'_j} \right) \ll_\eps q_\chi^{1/2+\eps} (d,q_\chi)^{1/2} \left( \frac{x}{q_{\chi} d} + 1 \right) .$$
We can write the left-hand side as
$$ \sum_{n: dn+a \in I} \chi(f(n))$$
where 
$$ f(n) \coloneqq \prod_{j\in J} \frac{dn+a+h_j}{d'_j}.$$
Suppose that there is a prime $p$ not dividing $d$ such that $f$ is a constant multiple of a square modulo $p$.  Then the roots $-\frac{a+h_j}{d}\ (p)$ of $f$ must experience a repetition, and hence $p$ divides $h_i-h_j$ for some $1 \leq i < j \leq k+\ell$.  Since the $h_1,\dots,h_{k+\ell}$ are fixed, this forces $p=O(1)$.  From the Chinese remainder theorem (and the fact that $q_\chi$ is a fundamental discriminant), we conclude that the largest factor $d'$ of $q_\chi$ for which $f$ is a constant multiple of a square modulo $d'$ is $O( (d,q_\chi))$. The claim now follows from~\eqref{sumni}.
\end{proof}

\subsection{Consequences of Kloosterman sum bounds}

We recall\footnote{For the applications in this paper one could also proceed using the weaker but more elementary bounds of Kloosterman~\cite{kloosterman}, as the important thing is that we gain a power savings over the trivial bound of $q$, at the cost of degrading the numerical exponent $10k$ in~\eqref{x-range} somewhat.  We leave the details of this variant of the argument to the interested reader.} Estermann's form~\cite{estermann}
$$
\left|\sum_{x \in \Z/q\Z: (x,q)=1} e_q(u_1 x + u_2 x^*)\right| \leq \tau(q) q^{1/2} (u_1,u_2,q)^{1/2}
$$
of the Weil bound for Kloosterman sums, where $x^*$ is the inverse of $x$ in $\Z/q\Z$ and $u_1,u_2$ are arbitrary integers.  From this and a simple change of variables we see that
\begin{equation}\label{estermann}
|\E_{n_1,n_2 \in \Z} 1_{w n_1 n_2 = a\ (q)} e_q(u_1 n_1 + u_2 n_2)| \leq \tau(q) q^{-3/2} (u_1,u_2,q)^{1/2}
\end{equation}
for any natural number $q$ and integers $w, a,u_1,u_2$ with $(w,q) = (a,q)=1$, where we use the averaging notation
$$ \E_{n_1,n_2 \in \Z} f(n_1,n_2) \coloneqq \frac{1}{L^2} \sum_{n_1=1}^L \sum_{n_2=1}^L f(n_1,n_2)$$
whenever $f \colon \Z^2 \to \C$ is a periodic function with some period $L$ (thus, $f(n_1+Lm_1, n_2+Lm_2) = f(n_1,n_2)$ for all integers $n_1,n_2,m_1,m_2$).

We will need to extend the bound~\eqref{estermann} to the case where $a$ shares a common factor with $q$, and where we also insert a periodic weight:

\begin{lemma}[Fourier coefficients on a hyperbola]\label{fourier-expand}  Let $q$ be a natural number, and let $a,u_1,u_2$ be integers.  Let $q_0$ be a factor of $q$ such that $(a,q) | q_0$.  Let $f \colon \Z^2 \to \C$ be a $1$-bounded\footnote{A function $f$ is $1$-bounded if $|f(x)| \leq 1$ for all $x$ in the domain of $f$.} function with period $q_0$. Then
$$
|\E_{n_1,n_2 \in \Z} f(n_1,n_2) 1_{n_1 n_2 = a\ (q)} e_q(u_1n_1 + u_2n_2)| \leq \tau(q_0)^2 q_0^{3/2} \tau(q) q^{-3/2} (u_1,u_2,q)^{1/2}.$$
\end{lemma}

The factor $\tau(q_0)^2 q_0^{3/2}$ can be improved somewhat, but we will not attempt to optimize it here.  This bound will be needed in step (iv) of the main argument, in order to dispose of the non-Type I portion $\Lambda_\Siegel^\flat$ to the Siegel approximant $\Lambda_\Siegel$.

\begin{proof}  If $n_1 n_2 = a\ (q)$, then from considering the prime factorisations of $n_1,n_2,a,q$ we see that $(n_1,q_0), (n_2,q_0)$ must be factors of $(a,q)$ and hence of $q_0$; also, we have $((n_1,q_0)(n_2,q_0), q) = (a,q_0) = (a,q)$.  Thus there are at most $\tau(q_0)^2$ possible choices for $(n_1,q), (n_2,q)$, and by the triangle inequality it suffices to show that
\begin{equation}\label{fnn}
|\E_{n_1,n_2 \in \Z} f(n_1,n_2) 1_{n_1 n_2 = a\ (q)} e_q(u_1n_1 + u_2n_2)| \leq q_0^{3/2} \tau(q) q^{-3/2} (u_1,u_2,q)^{1/2}.
\end{equation}
under the additional hypothesis that $f$ is supported in the region where $(n_1,q_0) = q_1$, $(n_2,q_0)=q_2$ for some factors $q_1,q_2$ of $q_0$ with
\begin{equation}\label{q1q2}
 (q_1 q_2, q) = (a,q).
\end{equation}
In particular, if we write $q' \coloneqq \frac{q}{(a,q)}$, then the quantity $w = \frac{q_1 q_2}{(a,q)}$ is a primitive element of $\Z / q'
\Z$. Making the change of variables $n_1 = q_1 n'_1$, $n_2 = q_2 n'_2$, we can now rewrite the left-hand side of~\eqref{fnn} as
$$ \frac{1}{q_1 q_2} |\E_{n'_1,n'_2 \in \Z} f(q_1 n'_1,q_2 n'_2) 1_{w n'_1 n'_2 = \frac{a}{(a,q)}\ (q')} e_q(u_1q_1 n_1 + u_2q_2 n_2)|.$$
By Fourier inversion and the Plancherel formula we have
$$f(q_1 n'_1,q_2 n'_2) = \sum_{k_1 \in \Z/(q_0/q_1)\Z} \sum_{k_2 \in \Z/(q_0/q_2)\Z} c_{k_1,k_2} e_{q_0}( k_1 q_1 n'_1 + k_2 q_2 n'_2 )$$
where the coefficients $c_{k_1,k_2}$ obey the bound
$$ \sum_{k_1 \in \Z/(q_0/q_1)\Z} \sum_{k_2 \in \Z/(q_0/q_2)\Z} |c_{k_1,k_2}|^2 \leq 1,$$
and hence by Cauchy--Schwarz
$$ \sum_{k_1 \in \Z/(q_0/q_1)\Z} \sum_{k_2 \in \Z/(q_0/q_2)\Z} |c_{k_1,k_2}| \leq \frac{q_0}{q_1^{1/2} q_2^{1/2}}.$$
Thus by the triangle inequality and pigeonhole principle, we can bound the left-hand side of~\eqref{fnn} by
$$ \frac{q_0}{q^{3/2}_1 q^{3/2}_2} \left|\E_{n'_1,n'_2 \in \Z} 1_{w n'_1 n'_2 = \frac{a}{(a,q)}\ (q')} e_q\left(\left(u_1+k_1 \frac{q}{q_0}\right)q_1 n_1 + \left(u_2+k_2 \frac{q}{q_0}\right)q_2 n_2\right)\right|$$
for some integers $k_1,k_2$.  Since $1_{w n'_1 n'_2 = \frac{a}{(a,q)}\ (q')}$ is a $q'$-periodic function of $n'_1,n'_2$, this expression vanishes unless the integers $(u_1+k_1 \frac{q}{q_0})q_1$, $(u_2+k_2 \frac{q}{q_0})q_2$ are divisible by $q/q' = (a,q)$.  Since $w$ and $\frac{a}{(a,q)}$ are both primitive in $\Z/q'\Z$, we may then apply~\eqref{estermann} and bound the left-hand side of~\eqref{fnn} by
$$ \frac{q_0}{q^{3/2}_1 q^{3/2}_2} (q')^{-3/2} \left(\left(u_1+k_1 \frac{q}{q_0}\right)\frac{q_1}{(a,q)}, \left(u_2+k_2 \frac{q}{q_0}\right) \frac{q_2}{(a,q)}, q'\right)^{1/2}$$
which we can rewrite as
$$ \frac{q_0}{q^{3/2}_1 q^{3/2}_2} (a,q) q^{-3/2} d^{1/2}$$
where
$$d \coloneqq \left( \left(u_1+k_1 \frac{q}{q_0}\right)q_1, \left(u_2+k_2 \frac{q}{q_0}\right) q_2, q\right).$$
By construction, we have
$$\left(u_1+k_1 \frac{q}{q_0}\right)q_1 =  \left(u_2+k_2 \frac{q}{q_0}\right) q_2 = q = 0\ (d)$$
and hence by taking suitable linear combinations
$$ u_1 q_0 q_1 q_2 = u_2 q_0 q_1 q_2 = q q_0 q_1 q_2 = 0\ (d).$$
We conclude in particular that $d | q_0 q_1 q_2 (u_1,u_2,q)$, and the claim follows (noting from~\eqref{q1q2} that $(a,q) \leq q_1 q_2$).
\end{proof}

From Lemma~\ref{fourier-expand} and the Fourier inversion formula one can express the periodic function $f(n_1,n_2) 1_{n_1 n_2 = a\ (q)}$ as a linear combination of Fourier phases $e_q(u_1n_1 + u_2n_2)$ with good bounds on the Fourier coefficients.  However, the contribution of those terms in which one of $u_1,u_2$ is divisible by $q$ (or by a very large factor of $q$) will be inconvenient to handle.  We therefore perform the following substitute expansion:

\begin{lemma}[Modified Fourier expansion]\label{mfe} Let $q$ be a natural number, and let $a$ be an integer.  Let $q_0$ be a factor of $q$ such that $(a,q) | q_0$.  Let $f \colon \Z^2 \to \C$ be a $1$-bounded function with period $q_0$.  Define $q'_0 \coloneqq (q_0 (a,q), q)$. Then we have
\begin{align*} f(n_1,n_2) 1_{n_1 n_2 = a\ (q)} &= \frac{\alpha q'_0}{q} f(n_1,n_2) 1_{n_1 n_2 = a\ (q'_0)} 1_{(n_1 n_2,q)=(a,q)}\\
&\quad + \sum_{u_1,u_2 \in \Z/q\Z: \frac{q}{q_0} \nmid u_1, u_2} c_{u_1,u_2} e_q(u_1 n_1 + u_2 n_2) 
\end{align*}
where $\alpha$ is the quantity
$$ \alpha \coloneqq \prod_{p | \frac{q}{q'_0}; p \nmid \frac{q'_0}{(a,q)}} \frac{p}{p-1}$$
and the coefficients $c_{u_1,u_2}$ obey the bounds
$$
|c_{u_1,u_2}| \leq 2 \tau(q_0)^2 q_0^{3/2} \tau(q) q^{-3/2} (u_1,u_2,q)^{1/2}.$$
\end{lemma}

\begin{proof}  We may assume without loss of generality that $1 \leq a \leq q$. Let ${\mathcal A}$ denote the collection of those $1 \leq a' \leq q$ such that $a' = a\ (q'_0)$ and $(a',q) = (a,q)$.  From Lemma~\ref{fourier-expand} we see that the Fourier coefficient
\begin{equation}\label{fc}
 \E_{n_1,n_2 \in \Z} f(n_1,n_2) (1_{n_1 n_2 = a\ (q)}-1_{n_1 n_2 = a'\ (q)}) e_q(u_1n_1 + u_2n_2)
\end{equation}
for $u_1,u_2 \in \Z/q\Z$ is bounded in magnitude by $2\tau(q_0)^2 q_0^{3/2} \tau(q) q^{-3/2} (u_1,u_2,q)^{1/2}$ for any $a' \in {\mathcal A}$.  We claim furthermore that this Fourier coefficient vanishes whenever one of $u_1,u_2$ is divisible by $q/q_0$.  Indeed, suppose for instance that $u_2$ is divisible by $q/q_0$, so that $n_2 \mapsto f(n_1 n_2) e_q(u_1 n_1 + u_2 n_2)$ is $q_0$-periodic for any $n_1$.  To obtain the vanishing of~\eqref{fc}, it suffices to show that
\begin{equation}\label{nq}
 \sum_{n_2 \in \Z/q\Z: n_2 = a_2\ (q_0)} 1_{n_1 n_2 = a\ (q)} = \sum_{n_2 \in \Z/q\Z: n_2 = a_2\ (q_0)} 1_{n_1 n_2 = a'\ (q)} 
\end{equation}
for any integers $n_1,a_2$.  But since $(a',q)=(a,q)$, we can write $a' = w a\ (q)$ for some primitive $w \in \Z/q\Z$; since $a'=a\ ((q_0 (a,q),q))$
we have $w = 1\ ((q_0,q/(a,q)))$; as we have the freedom to adjust $w$ by an arbitrary multiple of $q/(a,q)$ we may in fact assume that $w = 1\ (q_0)$.  The claim~\eqref{nq} then follows after applying the change of variables $n_2 \mapsto w n_2$ on the right-hand side. We argue similarly if $u_1$ is divisible by $q/q_0$ instead of $u_2$.

Averaging in $a'$, we conclude that the Fourier coefficient
$$ \E_{n_1,n_2 \in \Z} f(n_1,n_2) (1_{n_1 n_2 = a\ (q)}-\E_{a' \in {\mathcal A}} 1_{n_1 n_2 = a'\ (q)}) e_q(u_1n_1 + u_2n_2)$$
is bounded in magnitude by $2\tau(q_0)^2 q_0^{3/2} \tau(q) q^{-3/2} (u_1,u_2,q)^{1/2}$, and vanishes whenever $u_1$ or $u_2$ vanish in $\Z/q\Z$.  To establish the claim, it now suffices by the Fourier inversion formula to obtain the identity
$$
  \E_{a' \in {\mathcal A}} 1_{n = a'\ (q)} = \frac{\alpha q'_0}{q} 1_{n = a\ (q'_0)} 1_{(n,q)=(a,q)}$$
for any integer $n$.  By the Chinese remainder theorem, it suffices to establish this identity at each prime $p$, that is to say it suffices to show that
$$ \E_{a' = a\ (p^{j_0}): (a',p^j) = (a,p^j)} 1_{n=a'\ (p^j)} = \frac{\alpha_p p^{j_0}}{p^j} 1_{n=a\ (p^{j_0})} 1_{(n,p^j) = (a,p^j)}$$
whenever $p$ is a prime, $0 \leq j_0 \leq j$, and $a$ is an integer with $(a,p^j) | p^{j_0}$, where $\alpha_p \coloneqq \frac{p}{p-1}$ if $j > j_0$ and $(a,p^j)=p^{j_0}$, and $\alpha_p = 1$ otherwise.  But this follows by a direct case analysis:
\begin{itemize}
\item If $j=j_0$, then the conditions $(a',p^j) = (a,p^j)$ and $(n,p^j)=(a,p^j)$ are redundant, $\alpha_p=1$, $a'$ is restricted to a single residue class mod $p^j$, and both sides are equal to $1_{n = a\ (p^{j_0})}$.
\item If $j<j_0$ and $(a,p^j)<p^{j_0}$, then the conditions $(a',p^j) = (a,p^j)$ and $(n,p^j)=(a,p^j)$ are redundant, $\alpha_p=1$, $a'$ is restricted to $p^{j-j_0}$ residue classes mod $p^j$, and both sides are equal to $\frac{1}{p^{j-j_0}} 1_{n=a\ (p^{j_0})}$.
\item If $j<j_0$ and $(a,p^j)=p^{j_0}$, then $\alpha_p = \frac{p}{p-1}$, $a'$ is restricted to $\frac{p-1}{p} p^{j-j_0}$ residue classes mod $p^j$, and both sides are equal to $\frac{p}{p-1} \frac{1}{p^{j-j_0}} 1_{n=a\ (p^{j_0})}$.
\end{itemize}
\end{proof}

\section{First step: replacing the Liouville function with a Siegel model}\label{sec:siegel-liouville}

We now execute step (i) of the strategy outlined in the introduction.  From~\eqref{fmulti} we have the splitting
$$ \lambda = \lambda_{(\leq R)} * \lambda_{(>R)}.$$
In view of Corollary~\ref{excep-bound}, we expect $\lambda$ to resemble the exceptional character $\chi$ on the rough numbers $\N_{(>R)}$.  It is therefore natural to introduce the Siegel approximant
\begin{equation}\label{lsi-def}
 \lambda_\Siegel \coloneqq \lambda_{(\leq R)} * \chi_{(>R)},
\end{equation}
thus $\lambda_\Siegel$ is the completely multiplicative function that agrees with $\lambda$ for primes $p \leq R$ and agrees with $\chi$ for primes $p > R$.  Similar approximants were also introduced in~\cite{german-katai}, \cite{chinis}.
Clearly $\lambda$, $\lambda_\Siegel$ are both bounded by $1$:
\begin{equation}\label{ls-bound}
|\lambda(n)|, |\lambda_\Siegel(n)| \leq 1.
\end{equation}

The error between $\lambda$ and $\lambda_\Siegel$ can be controlled by exceptional primes and by rough numbers:

\begin{lemma}[Error bound between $\lambda$ and $\lambda_\Siegel$]\label{lam-err}  For any natural number $n \leq 2x$, one has
\begin{equation}\label{la}
 \lambda(n) - \lambda_\Siegel(n) \ll \sum_{p^*|n, R < p^* \leq x/R} 1 + \sum_{d \leq 2R: d|n} 1_{(\geq x/R)}(n/d).
\end{equation}
\end{lemma}

\begin{proof}  If $n$ is not divisible by any exceptional prime $p^*>R$, then we have $\lambda(n) = \lambda_\Siegel(n)$ since $\lambda, \lambda_\Siegel$ agree on every prime dividing $n$.  Clearly~\eqref{la} holds in this case.  If $n$ is divisible by an exceptional $R < p^* \leq x/R$, then the first term on the right-hand side of~\eqref{la} is at least one, and the claim~\eqref{la} then follows from~\eqref{ls-bound}.

The only remaining case is if $n$ is divisible by an exceptional prime $p^* \geq x/R$, so $n = p^* d$ for some $d \leq 2R$. Since $n/d = p^* \geq x/R$ is prime, the second term on the right-hand side of~\eqref{la} is at least one, and the claim~\eqref{la} again follows from~\eqref{ls-bound}.
\end{proof}

In this section we establish

\begin{proposition}[Replacing $\lambda$ with a Siegel model]\label{step-1}  We have
\begin{align*}
& \E_{n \leq x} \Lambda(n+h_1) \cdots \Lambda(n+h_k) \lambda(n+h'_1) \cdots \lambda(n+h'_\ell) \\
&\quad \approx \E_{n \leq x} \Lambda(n+h_1) \cdots \Lambda(n+h_k) \lambda_\Siegel(n+h'_1) \cdots \lambda_\Siegel(n+h'_{\ell}).
\end{align*}
\end{proposition}

From~\eqref{ls-bound} and the triangle inequality, it suffices to show that
$$ \E_{n \leq x} \Lambda(n+h_1) \cdots \Lambda(n+h_k) |\lambda(n+h'_{j'}) - \lambda_\Siegel(n+h'_{j'})|
\approx 0$$
for each $1 \leq j' \leq \ell$.  Applying Lemma~\ref{lam-err}, it suffices to show the bounds
\begin{equation}\label{p-1}
\sum_{R < p^* \leq x/R} \E_{n \leq x} \Lambda(n+h_1) \cdots \Lambda(n+h_k) 1_{p^*|n+h'_{j'}}
\approx 0
\end{equation}
and
\begin{equation}\label{p-2}
\sum_{d \leq 2R}  \E_{n \leq x} \Lambda(n+h_1) \cdots \Lambda(n+h_k) 1_{d|n+h'_{j'}} 1_{(\geq x/R)}\left(\frac{n+h'_{j'}}{d}\right)
\approx 0.
\end{equation}

We begin with~\eqref{p-1}.  For $n \leq x$ and $1 \leq j \leq k$, the quantity $\Lambda(n+h_{j})$ is bounded by $\log(2x) 1_{(\geq \sqrt{2x})}(n+h_j)$ unless we are in the exceptional case where $n+h_{j}$ is of the form $p^i$ for some prime $p < \sqrt{2x}$ (cf. the sieve of Eratosthenes).  The contribution of such exceptional cases can easily be shown to be $\approx 0$, so it suffices to show that
$$ 
(\log^k x) 
\sum_{R < p^* \leq x/R} \E_{n \leq x} 1_{(\geq \sqrt{2x})}(n+h_1) \cdots 1_{(\geq \sqrt{2x})}(n+h_k) 1_{p^*|n+h'_{j'}}
\approx 0.$$
Let $p^*$ be as in the above sum.  Changing variables, we have
\begin{align*} &\E_{n \leq x} 1_{(\geq \sqrt{2x})}(n+h_1) \cdots 1_{(\geq \sqrt{2x})}(n+h_k) 1_{p^*|n+h'_{j'}}\\
&\ll \frac{1}{p^*} \E_{n \leq 2x/p^*} 1_{(\geq \sqrt{2x})}(p^* n+h_1-h'_{j'}) \cdots 1_{(\geq \sqrt{2x})}(p^* n+h_k-h'_{j'}).
\end{align*}
Let $C_0$ be a sufficiently large constant depending on $h_1,\dots,h_k,h'_{j'}$.  Then for any prime $C_0 < p \leq \sqrt{2x}$ other than $p^*$, the support set of $1_{(\geq \sqrt{2x})}(p^* n+h_1-h'_{j'}) \cdots 1_{(\geq \sqrt{2x})}(p^* n+h_k-h'_{j'})$ excludes $k$ residue classes modulo $p$.  Thus by\footnote{One could also use Lemma~\ref{nusieve} here instead if desired to give a comparable estimate.} Lemma~\ref{sub} we have
$$ \E_{n \leq x} 1_{(\geq \sqrt{2x})}(n+h_1) \cdots 1_{(\geq \sqrt{2x})}(n+h_k) 1_{p|n+h'_{j'}} \ll \frac{1}{p^*} \prod_{C_0 < p \leq \min( 2x/p^*, \sqrt{2x}): p \neq p^*} \left( 1 - \frac{k}{p}\right)$$
and hence by Mertens' theorem~\eqref{mertens-3} and the bound $p_* \leq x/R$
$$ (\log^k x) \E_{n \leq x} 1_{(\geq \sqrt{2x})}(n+h_1) \cdots 1_{(\geq \sqrt{2x})}(n+h_k) 1_{p^*|n+h'_{j'}} \ll \frac{\log^k_R x}{p^*}.$$
The claim~\eqref{p-1} now follows from Corollary~\ref{excep-bound} and~\eqref{llog}.

Now we prove~\eqref{p-2}.  Arguing as in the proof of~\eqref{p-1}, it suffices to show that
\begin{equation}\label{p-2-alt}
(\log^k x) \sum_{d \leq 2R}  \E_{n \leq x} 1_{(\geq \sqrt{2x})}(n+h_1) \cdots 1_{(\geq \sqrt{2x})}(n+h_k) 1_{d|n+h'_{j'}} 1_{(\geq x/R)}\left(\frac{n+h'_{j'}}{d}\right)
\approx 0.
\end{equation}
For $d \leq 2R$, we have after change of variables that
\begin{align*}
&\E_{n \leq x} 1_{(\geq \sqrt{2x})}(n+h_1) \cdots 1_{(\geq \sqrt{2x})}(n+h_k) 1_{d|n+h'_{j'}} 1_{(\geq x/R)}\left(\frac{n+h'_{j'}}{d}\right)\\
&\quad \ll \frac{1}{d}
\E_{n \leq 2x/d} 1_{(\geq \sqrt{2x})}(dn+h_1-h'_{j'}) \cdots 1_{(\geq \sqrt{2x})}(dn+h_k-h'_{j'}) 1_{(\geq x/R)}(n).
\end{align*}
With $C_0$ as before, we see that for any prime $C_0 \leq p < \sqrt{2x}$ not dividing $d$ we are excluding $k+1$ residue classes modulo $p$ (since $h'_{j'}$ is distinct from $h_1,\dots,h_k$), hence by Lemma~\ref{sub}
$$
\E_{n \leq x} 1_{(\geq \sqrt{2x})}(n+h_1) \cdots 1_{(\geq \sqrt{2x})}(n+h_k) 1_{d|n+h'_{j'}} 1_{(\geq x/R)}\left(\frac{n+h'_{j'}}{d}\right)
\ll \frac{1}{d} \prod_{C_0 \leq p < \sqrt{2x}: p \nmid d} \left(1 - \frac{k+1}{p}\right)$$
and hence by Mertens' theorem~\eqref{mertens-3}
\begin{align*}
&(\log^k x) \E_{n \leq x} 1_{(\geq \sqrt{2x})}(n+h_1) \cdots 1_{(\geq \sqrt{2x})}(n+h_k) 1_{d|n+h'_{j'}} 1_{(\geq x/R)}\left(\frac{n+h'_{j'}}{d}\right)\\
&\ll \frac{1}{d\log x} \prod_{p|d} \left(1 + O\left(\frac{1}{p}\right)\right).
\end{align*}
By~\eqref{euler-stop} we may therefore bound the left-hand side of~\eqref{p-2-alt} by
$$ \frac{1}{\log x} \prod_{p \leq 2R} \left(1 + \frac{1}{p} + O\left(\frac{1}{p^2}\right)\right).$$
By~\eqref{mertens-3} this latter quantity is $O( \log_x R )$, and the claim follows from~\eqref{lrn}.

\section{Second step: replacing the von Mangoldt function with a Siegel model}\label{sec:siegel-vonmangoldt}

We now execute step (ii) of the strategy outlined in the introduction.  In order to (mostly) restrict to rough numbers, we will insert the Selberg sieve $\nu$ defined in~\eqref{nuz-def}.
Namely, observe that
$$ \Lambda - \Lambda \nu$$
is supported on prime powers $p^j$ with $p \leq R$ and can be crudely bounded by $O(\log^2 x)$ on such powers.  Since the number of such powers of size $O(x)$ is crudely bounded by $O( R \log x )$, one easily sees from the triangle inequality that
\begin{equation}\label{insert-nu}
\begin{split}
& \E_{n \leq x} \Lambda(n+h_1) \cdots \Lambda(n+h_k) \lambda_\Siegel(n+h'_1) \cdots \lambda_\Siegel(n+h'_\ell)\\
&\quad \approx
\E_{n \leq x} \Lambda \nu(n+h_1) \cdots \Lambda \nu(n+h_k) \lambda_\Siegel(n+h'_1) \cdots \lambda_\Siegel(n+h'_\ell) 
\end{split}
\end{equation}
(with plenty of room to spare in the error term).  Next, we expand
$$ \Lambda \nu = (\mu * \log) \nu.$$
Since $\mu$ is expected to be close to $\chi$ on rough numbers, and the Selberg sieve $\nu$ is mostly restricted to such numbers, it is then natural to introduce the Siegel approximant
$$ \Lambda_\Siegel \coloneqq (\chi * \log) \nu.$$

From the triangle inequality we have the crude bounds
\begin{equation}\label{tri}
 \Lambda \nu(n), \Lambda_\Siegel(n) \ll \tau \nu(n) \log x
\end{equation}
We also have the following bound for the error between $\Lambda \nu$ and $\Lambda_\Siegel$:

\begin{lemma}  For $n \leq 2x$, we have the bounds
\begin{equation}\label{lasn}
 \Lambda \nu(n) - \Lambda_\Siegel(n) \ll E(n) + F(n) + G(n)
\end{equation}
where
\begin{align}
E(n) &\coloneqq \left( \sum_{R_0 < p^* \leq \sqrt{2x}} 1_{p^*|n} + \sum_{R_0 < p \leq \sqrt{2x}} 1_{p^2|n}\right) \tau \nu(n) \log x\label{E-def}\\
F(n) &\coloneqq \left(\sum_{1 < d \leq D:d|n} \tau_{(\leq R_0)}(d)^{O(1)}\right) \nu(n) \log x.\label{F-def}\\
G(n) &\coloneqq \sum_{\sqrt{2x} < p^* \leq 2x/R^{1/2}} 1_{p^*|n} \Lambda \nu(n/p^*).\label{G-def}
\end{align}
\end{lemma}
 
\begin{proof} If $n$ is divisible by an exceptional $R_0 < p^* \leq \sqrt{2x}$, then $E \gg \tau(n) \nu(n) \log x$, and~\eqref{lasn} then follows from~\eqref{tri} and~\eqref{E-def}.  Similarly if $n$ is divisible by the square of a prime $p > R_0$ (which must then necessarily be at most $\sqrt{2x}$).

Next, suppose that $n$ is not divisible by any exceptional prime $p^* > R_0$, nor by any square $p^2$ of a prime $p>R_0$.  We write
$$ \chi * \log = (1*\chi)*\mu*\log = (1*\chi)*\Lambda.$$
Note that $1*\chi(d)$ is only non-zero when $d$ is the product of exceptional primes times a perfect square, so if $d|n$ and $n$ is as above then $d$ must be the product of some primes less than or equal to $R_0$.  Also $\sum_{d|n} \Lambda(d) = \log n$.  Thus, for $n$ as above, we have
$$ \chi*\log(n) \leq \tau(n_{(\leq R_0)}) \log n,$$
where we recall that $n_{(\leq R_0)}$ is the largest factor of $n$ that is the product of primes less than or equal to $R_0$.  Applying~\eqref{landreau-special} (with $n$ replaced by $n_{(\leq R_0)}$), we have
$$ \tau(n_{(\leq R_0)}) \ll \sum_{1 < d \leq D:d|n} \tau_{(\leq R_0)}(d)^{O(1)}$$
and the claim~\eqref{lasn} now follows in this case from~\eqref{F-def}.

We are left with the case where $n$ is divisible by an exceptional prime $p^* > \sqrt{2x}$. Then $n = dp^*$ for some $d < \sqrt{2x}$.
The only factors of $n$ that are less than or equal to $R$ are factors of $d$ as well, thus $\nu(n) = \nu(d)$.  Since
$$ \chi * \log(n) = \chi * 1 * \mu * \log(n) = \chi * 1 * \Lambda(n)$$
and $\chi*1$ vanishes at all factors of $n$ except for $1$ and $p^*$, we have
$$ \chi * \log(n) = \Lambda(n) + (1+\chi(p^*)) \Lambda(d)$$
and thus
$$  \Lambda(n) - \Lambda_\Siegel(n)  \ll \Lambda \nu(d).$$
If $p^* > 2x/R^{1/2}$ then $d \leq R^{1/2}$, and hence $\nu(d)$ vanishes by~\eqref{nuz-def}.  The claim~\eqref{lasn} now follows in this case from~\eqref{G-def}.

\end{proof}

Now we can prove

\begin{proposition}[Replacing $\Lambda$ with a Siegel model]\label{step-2}  We have
\begin{align*}
&\E_{n \leq x} \Lambda(n+h_1) \cdots \Lambda(n+h_k) \lambda_\Siegel(n+h'_1) \cdots \lambda_\Siegel(n+h'_\ell)\\ 
&\quad \approx \E_{n \leq x} \Lambda_\Siegel(n+h_1) \cdots \Lambda_\Siegel(n+h_k) \lambda_\Siegel(n+h'_1) \cdots \lambda_\Siegel(n+h'_{\ell}).
\end{align*}
\end{proposition}

In view of~\eqref{insert-nu} it suffices to show that
\begin{align*}
&\E_{n \leq x} \Lambda \nu (n+h_1) \cdots \Lambda \nu(n+h_k) \lambda_\Siegel(n+h'_1) \cdots \lambda_\Siegel(n+h'_\ell) \\
&\quad \approx \E_{n \leq x} \Lambda_\Siegel(n+h_1) \cdots \Lambda_\Siegel(n+h_k) \lambda_\Siegel(n+h'_1) \cdots \lambda_\Siegel(n+h'_\ell).
\end{align*}
By~\eqref{ls-bound} and the triangle inequality it suffices to show that
$$ \E_{n \leq x} |\Lambda \nu (n+h_1) \cdots \Lambda \nu(n+h_k) - \Lambda_\Siegel(n+h_1) \cdots \Lambda_\Siegel(n+h_k)| \approx 0.$$
From~\eqref{lasn} we have
$$ \Lambda_\Siegel(n+h_j) = \Lambda \nu(n+h_j) + O(E(n+h_j) + F(n+h_j) + G(n+h_j))$$
for $j=1,\dots,k$.   Multiplying these estimates together, we conclude that
\begin{align*}
\Lambda_\Siegel(n+h_1) \cdots \Lambda_\Siegel(n+h_k) &= \Lambda \nu (n+h_1) \cdots \Lambda \nu(n+h_k) \\
&\quad +O\left( \sum_{j=1}^k E(n+h_j) \prod_{1 \leq j' \leq k: j' \neq j} (\Lambda \nu + E + F + G)(n+h_{j'})\right) \\
&\quad + O\left(\sum_{j=1}^k F(n+h_j) \prod_{1 \leq j' \leq k: j' \neq j} (\Lambda \nu + F + G)(n+h_{j'})\right) \\
&\quad + O\left(\sum_{j=1}^k G(n+h_j) \prod_{1 \leq j' \leq k: j' \neq j} (\Lambda \nu + G)(n+h_{j'})\right).
\end{align*}
By the triangle inequality and relabeling, it thus suffices to establish the bounds
\begin{equation}\label{q-1}
 \E_{n \leq x} E(n+h_1) \prod_{j=2}^k (\Lambda \nu+E+F+G)(n+h_j) \approx 0
\end{equation}
and
\begin{equation}\label{q-3}
 \E_{n \leq x} F(n+h_1) \prod_{j=2}^k (\Lambda \nu+F+G)(n+h_j) \approx 0
\end{equation}
and
\begin{equation}\label{q-2}
 \E_{n \leq x} G(n+h_1) \prod_{j=2}^k (\Lambda \nu+G)(n+h_j) \approx 0.
\end{equation}

We begin with~\eqref{q-2}, which is a variant of~\eqref{p-2}.  We can bound $(\Lambda \nu + G)(n+h_j)$ by $O( \log(2x) 1_{(\geq R^{1/4})}(n+h_j) )$,
and we also have the bound
$$G(n+h_1) \ll \log(2x) \sum_{\sqrt{2x} < p^* \leq 2x/R^{1/2}} 1_{p^*|n+h_1} 1_{(\geq R^{1/4})}(n+h_1)$$
unless $n+h_j$ is of the form $p^m$ for some $p<R^{1/4}$ and $m \geq 1$, or $p' p^m$ for some $p<R^{1/4}$, $m \geq 2$, and $\sqrt{2x} \leq p' \leq 2x/R^{1/2}$.  

There are only $O( x \log x/R^{1/4} )$ such exceptional values of $n$ and their contribution is easily seen to be negligible using~\eqref{divisor-bound}.  Thus it will suffice to show that
\begin{equation}\label{q-2-alt}
 (\log^{k} x) \sum_{\sqrt{2x} < p^* \leq 2x/R^{1/2}}  \E_{n \leq x} 1_{p^*|n+h_1} \prod_{j=1}^k 1_{(\geq R^{1/4})}(n+h_j) \approx 0.
\end{equation}
Making the change of variables $n = p^* n' - h_1$ and using Lemma~\ref{sub} and Mertens' theorem~\eqref{mertens-3}, we see that
$$
\E_{n \leq x} 1_{p^*|n+h_1} \prod_{j=1}^k 1_{(\geq R^{1/4})}(n+h_j)  \ll \frac{1}{p^* \log^k R}.$$
The claim~\eqref{q-2-alt} now follows from Corollary~\ref{excep-bound} and~\eqref{R-def}.

Now we turn to~\eqref{q-1}.  Observe using~\eqref{landreau-special} that
$$ (\Lambda \nu+E+F+G)(n+h_j) \ll \left(\sum_{d_j \leq D: d_j|n+h_j} \tau(d_j)^{O(1)}\right) \nu(n+h_j) \log x$$
and so we can bound the left-hand side of~\eqref{q-1} by
\begin{equation}\label{rsu}
\ll (\log^k x) \left( \sum_{R_0 < p^* \leq \sqrt{2x}} a_{p^*} + \sum_{R_0 < p_0 \leq \sqrt{2x}} a_{p_0^2}\right)
\end{equation}
where
$$ a_d \coloneqq \sum_{d_1,\dots,d_k \leq D} \tau(d_1 \cdots d_k)^{O(1)} \E_{n \leq x} 1_{d|n+h_1} \prod_{j=1}^k 1_{d_j|n+h_j} \nu(n+h_j).$$
Let $d$ be equal to $p_0$ or $p_0^2$ for some prime $p_0 \leq \sqrt{2x}$.  If $d > \sqrt{2x}$, then $d = p_0^2$ for some $p_0 \gg x^{1/4}$.  From~\eqref{divisor-bound} one has the crude bound $a_d \ll_\eps \frac{x^\eps}{d} = \frac{x^\eps}{p_0^2}$ in this case, which certainly gives a negligible contribution.  Hence we may assume that $d \leq \sqrt{2x}$.
 
Applying Lemma~\ref{nusieve} and~\eqref{divisor-bound}, we have
$$ a_d \ll_{A} \frac{1}{\log^k R} \int_1^\infty f(\sigma) \frac{d\sigma}{\sigma^A} + \frac{D^{k+1} R^{2k}}{x}$$
for any $A>0$, where 
$$ f(\sigma) \coloneqq \sum_{d_1,\dots,d_k \leq \sqrt{2x}}  \frac{\tau([d,d_1] d_2 \cdots d_k)^{O(1)}}{[d,d_1] d_2 \cdots d_k} \left(\prod_{p|d d_1 \cdots d_k} \min(\sigma \log_R p,1)\right).$$
Using Euler products~\eqref{euler-stop} we can bound 
$$f(\sigma) \leq \prod_{p \leq \sqrt{2x}} E_p(\sigma)$$ 
where
$$ E_p(\sigma) \coloneqq \sum_{d_1,\dots,d_k \in \N_{(p)}}  \frac{\tau([d_{(p)},d_1] d_2 \cdots d_k)^{O(1)}}{[d_{(p)},d_1] d_2 \cdots d_k} \min(\sigma \log_R p,1)^{1_{p|d d_1 \cdots d_k} }.$$

If $p \neq p_0$, then $d_{(p)}=1$, and we can calculate
$$ E_p(\sigma) \leq 1 + O\left( \frac{\min(\sigma \log_R p,1)}{p} \right).$$
From~\eqref{mertens-alt} we then have
$$ \prod_{p \leq \sqrt{2x}} E_p(\sigma) \ll (1 + \sigma \log_R \sqrt{2x})^{O(1)} E_{p_0}(\sigma) \ll (\log_R^{O(1)} x) \sigma^{O(1)} E_{p_0}(\sigma).$$
Also, we have the crude bound
$$ E_{p_0}(\sigma) \ll \frac{\tau(d)^{O(1)}}{d}.$$
Putting all these estimates together, and choosing $A$ large enough, we conclude that
$$ a_d \ll \tau(d)^{O(1)}\frac{\log_R^{O(1)} x}{d \log^k R} + \frac{D^{k+1} R^{2k}}{x}.$$
Inserting this into~\eqref{rsu} and using Corollary~\ref{excep-bound}, \eqref{llog}, \eqref{rdk0}, we obtain the claim~\eqref{q-1}.

Finally, we establish~\eqref{q-3}. Observe from~\eqref{F-def}, \eqref{G-def} that
$$ 
(\Lambda \nu + F + G)(n+h_{j'}) \ll \left(\sum_{d_j \leq D:d_j|n+h_j} \tau_{(\leq R_0)}(d_j)^{O(1)}\right) \nu(n+h_j) \log x
$$
and so it suffices to show that
$$ (\log^k x) \sum_{d_1,\dots,d_k \leq D: d_1 > 1} \left(\prod_{j=1}^k \tau_{(\leq R_0)}(d_j)^{O(1)}\right)
\E_{n \leq x} \prod_{j=1}^k 1_{d_j|n+h_j} \nu(n+h_j) \approx 0.$$
Applying Lemma~\ref{nusieve} and~\eqref{divisor-bound}, we may estimate the left-hand side as
\begin{align*}
&\ll_{A} (\log_R^k x) \sum_{d_1,\dots,d_k \leq D: d_1 > 1} \frac{\prod_{j=1}^k \tau_{(\leq R_0)}(d_j)^{O(1)}}{d_1 \cdots d_k} \int_1^\infty \prod_{p|d_1 \cdots d_k} \min(\sigma \log_R p,1)\ \frac{d\sigma}{\sigma^A} \\
&\quad + \frac{R^{2k} D^{k+1} (\log^k x)}{x}
\end{align*}
for any $A>0$.  The second term is $\approx 0$ by~\eqref{rdk0}.  Replacing the condition $d_1>1$ by $(d_1,\dots,d_k) \neq (1,\dots,1)$, removing the constraints $d_1,\dots,d_k \leq D$, and factoring the Euler product using~\eqref{euler-stop}, the first term can be bounded by
\begin{equation}\label{lrx}
 (\log_R^k x) \int_1^\infty (\prod_{p \leq R_0} \tilde E_p(\sigma) - 1) \frac{d\sigma}{\sigma^A}, 
\end{equation}
where
$$ \tilde E_p(\sigma) \coloneqq \sum_{d_1,\dots,d_k \in \mathbb{N}_{(p)} \rangle}
\frac{\prod_{j=1}^k \tau(d_j)^{O(1)}}{d_1 \cdots d_k} \min(\sigma \log_R p,1)^{1_{p|d_1 \cdots d_k}}.$$
Direct calculation gives
$$ \tilde E_p(\sigma) = 1 + O\left( \frac{\min(\sigma \log_R p, 1)}{p} \right).$$
From~\eqref{mertens-alt} we have
$$ \prod_{p \leq R_0} \tilde E_p(\sigma) \leq (1 + \sigma \log_R R_0)^{O(1)}$$
and hence
$$ \left(\prod_{p \leq R_0} \tilde E_p(\sigma)\right) - 1 \ll \sigma^{O(1)} \log_R R_0.$$
We can thus bound~\eqref{lrx} for $A$ large enough by
$$ \ll (\log_R^k x) \log_R R_0$$
which is $\approx 0$ by~\eqref{rxr0}.  This concludes the proof of~\eqref{q-2} and hence of Proposition~\ref{step-2}.

\section{Third step: replacing the Liouville Siegel model with a Type I approximant}\label{sec:typei-liouville}

We now execute step (iii) of the strategy outlined in the introduction.  From~\eqref{lsi-def}, \eqref{respect}, \eqref{fmulti} and M\"obius inversion we have
\begin{align*}
\lambda_\Siegel &= \lambda_{(\leq R)} * (\mu \chi)_{(\leq R)} * \chi_{(\leq R)} * \chi_{(>R)} \\
&= (\lambda * \mu \chi)_{(\leq R)} * \chi.
\end{align*}
We now split
\begin{equation}\label{ls-split}
\lambda_\Siegel = \lambda_\Siegel^\sharp + \lambda_\Siegel^\flat
\end{equation}
where $\lambda_\Siegel^\sharp$ is the Type I approximant
\begin{equation}\label{ls-def}
\lambda_\Siegel^\sharp \coloneqq (\lambda * \mu \chi)_{(\leq R)} \psi_{\leq D} * \chi
\end{equation}
and $\lambda_\Siegel^\flat$ is the error
\begin{equation}\label{ls-flat-def}
\lambda_\Siegel^\flat \coloneqq (\lambda * \mu \chi)_{(\leq R)} \psi_{> D} * \chi.
\end{equation}
Here $\psi_{\leq D}, \psi_{>D}$ are the smooth cutoffs defined in~\eqref{psia}, \eqref{psibig-def}.  In particular we see that $\lambda_\Siegel(n) = \lambda_\Siegel^\sharp(n)$ whenever $n_{(\leq R)} \leq \sqrt{D}$.  Since $\sqrt{D}$ is significantly larger than $R$, and $R$-smooth numbers become extremely sparse at scales much larger than $R$, we thus see that $\lambda_\Siegel$, $\lambda_\Siegel^\sharp$ agree with each other for ``typical'' $n$, and would thus be heuristically expected to be close to each other; in other words, $\lambda_\Siegel^\flat$ would be expected to be small on average.

Unfortunately, $\lambda_\Siegel^\sharp, \lambda_\Siegel^\flat$ are not bounded.  However, we can still obtain a reasonable bound on the latter quantity:

\begin{lemma}\label{lamlam}  For any $n\leq 2x$, we have
\begin{equation}\label{ga}
 \lambda_\Siegel^\flat(n) \ll H(n)
\end{equation}
where
\begin{equation}\label{geo-def}
 H(n) \coloneqq \sum_{d \leq D: d|n} \alpha(d)
\end{equation}
and $\alpha(d)$ are non-negative quantities obeying the bounds
\begin{equation}\label{dao}
\sum_{d \leq D} \tau(d)^A \frac{\alpha(d)}{d} \ll_A \exp\left(-\frac{1}{8}\log_R D\right)
\end{equation}
for any $A \geq 1$.
\end{lemma}

\begin{remark} Note that by~\eqref{eta-decay} we have
\begin{align}\label{logRD}
\exp(-\frac{1}{8}\log_R D)\ll_A \log^{-A} \eta
\end{align}
for any $A \geq 1$. We shall need~\eqref{logRD} later, but we stated Lemma~\ref{lamlam} in a stronger form to emphasize that it does not use any information on exceptional characters.   
\end{remark}

\begin{proof}  From~\eqref{ls-flat-def} and the triangle inequality we have
\begin{equation}\label{lambda-diff}
 \lambda_\Siegel^\flat(n) = \beta * \chi_{(>R)}(n) \ll |\beta| * 1_{(>R)}(n)
\end{equation}
where
\begin{equation}\label{beta-def}
 \beta  \coloneqq (\lambda * \mu \chi)_{(\leq R)} \psi_{> D} * \chi_{(\leq R)}.
\end{equation}
To control $\beta$ we perform a Fourier expansion on $\psi_{>D}$, which is the only term on the right-hand side of~\eqref{beta-def} which is not multiplicative.  Applying Fourier inversion~\eqref{fourier-inverse} to the function $g(u) \coloneqq e^{-u} (1 - \psi((\log_D R)u)$ and setting $u \coloneqq \log_{R} n$, we conclude the identity
\begin{equation}\label{psidn}
 \psi_{>D}(n) = \int_\R n^{\frac{1+it}{\log R}} f(t)\ dt
\end{equation}
where
$$ f(t) \coloneqq \frac{1}{2\pi} \int_0^\infty e^{-(1+it)x} (1 - \psi((\log_D R) x))\ dx.$$
From the triangle inequality we have
$$ f(t) \ll \exp(- \frac{1}{4} \log_{R} D ) \ll_A \log^{-A} \eta$$
for any $A>0$, while from repeated integration by parts we have 
$$ f(t) \ll_A \frac{1}{|t|^A}$$
for any positive integer $A$.  Combining the two bounds, we conclude that
\begin{equation}\label{varphi}
f(t) \ll_A \frac{\exp(-\frac{1}{8}\log_R D)}{(1+|t|)^A}
\end{equation}
for any $A>0$.

From~\eqref{beta-def}, \eqref{psidn} we have
$$ \beta(n) = \int_\R \beta_t(n) f(t)\ dt$$
where
\begin{equation}\label{betat-calc}
 \beta_t \coloneqq (\lambda * \mu \chi)_{(\leq R)} (\cdot)^{\frac{1+it}{\log R}} * \chi_{(\leq R)}.
\end{equation}
From~\eqref{lambda-diff} and the triangle inequality we then have
$$  \lambda_\Siegel^\flat(n) \leq \int_\R |\beta_t| * 1_{(>R)}(n) |f(t)|\ dt.$$

The function $|\beta_t|$ is multiplicative and supported on $\N_{(\leq R)}$, thus
$$ |\beta_t| * 1_{(>R)}(n) \ll \prod_{p \leq R} |\beta_t(n_{(p)})|.$$
From~\eqref{betat-calc} we see that 
$$ |\beta_t(p^j)| = 1$$
when $p \leq R$ and $\chi(p)=-1$ (because $\lambda * \mu \chi$ agrees with $\lambda * \mu \lambda = 1_{\{1\}}$ on $\N_{(p)}$), and
$$ |\beta_t(p^j)| = p^{\frac{j}{\log R}} = \exp( j \log_{R} p )$$
when $p \leq R$ and $\chi(p)=0$.  For $\chi(p)=+1$ the situation is more complicated: direct calculation gives
$$ \beta_t(p^j) = P_j(p^{\frac{1+it}{\log R}})$$
where $P_j$ is the polynomial
$$ P_j(z) \coloneqq 1 - 2z + 2z^2 - \cdots + (-1)^j 2 z^j.$$
Note that $|P_j(1)| \leq 1$ and $P'_j(z) \ll j^{O(1)} (1 + |z|^{j-1})$ for any $z$, hence by the fundamental theorem of calculus
$$ |P_j(z)| \leq 1 + O( |z-1| j^{O(1)} )$$
whenever $|z| \leq 1 + \frac{1}{j}$.  Also from the triangle inequality we have $|P_j(z)| \ll j |z|^j$ for $|z| \geq 1$.  We thus have
$$ |P_j(z)| \leq \min( 1+j^{O(1)} |z-1|,\exp( O(  j \log |z| +1+ \log j  ) ))$$
for $|z| \geq 1$.  Thus regardless of the value of $\chi(p)$, we have the upper bound
$$ |\beta_t(p^j)| \leq \exp( O( a_{t,p^j} ) )$$
for $p \leq R$, where $a_{t,p^j}$ is the quantity
$$ a_{t,p^j} \coloneqq \min( j^{C} (1+|t|) \log_R p, j \log_R p +1+ \log j )$$
for some large constant $C\geq 1$ and for all $j \geq 1$, with the convention $a_{t,1}=0$.  We conclude that
$$ |\beta_t| * 1_{(>R)}(n) \ll \exp\left( O\left( \sum_{p \leq R} a_{t,n_{(p)}} \right) \right).$$
To convert the right-hand side into Type I sums we apply Lemma~\ref{landreau-lemma}(i) to split
$$ n = n_{(> R)} n_1 \cdots n_m$$
where $m = O(1)$ and $n_1,\dots,n_m \leq D$ lie in $\N_{(\leq R)}$.  We then have
$$ n_{(p)} = (n_1)_{(p)} \cdots (n_m)_{(p)}$$
for all $p \leq R$, and hence
$$ a_{t,n_{(p)}} \ll \sum_{i=1}^m a_{t,(n_i)_{(p)}}.$$
Using the definition of $a_{t,p^j}$ and the inequality $(j_1+j_2)^{C}\ll_C j_1^C+j_2^C$, we conclude that
$$ |\beta_t| * 1_{(>R)}(n) \ll \exp\left( O\left( \sum_{i=1}^m \sum_{p \leq R} a_{t,(n_i)_{(p)}} \right) \right)$$
and hence (since $m=O(1)$), we have
\begin{align*}
|\beta_t| * 1_{(>R)}(n) &\ll \exp\left( O\left( \sum_{p \leq R} a_{t,(n_i)_{(p)}} \right) \right)\\
&=\prod_{p\leq R}\left(1+(\exp(O(a_{t,(n_i)_{(p)}}))-1)\right)
\end{align*}
for some $i=1,\dots,m$.  In particular, we see that
$$|\beta_t| * 1_{(>R)}(n) \ll \sum_{d \leq D: d|n} 1_{(\leq R)}(d) \prod_{p\leq R}\left(\exp( O(  a_{t,d_{(p)}} ) )-1\right).$$
We therefore obtain the bound~\eqref{ga} with
$$ \alpha(d) \coloneqq 1_{(\leq R)}(d) \int_\R \prod_{p\leq R}\left(\exp( O( a_{t,d_{(p)}} ) )-1\right) |f(t)|\ dt.$$
It remains to establish the bound~\eqref{dao}.  We use Fubini's theorem and Euler product expansion~\eqref{euler-stop} to bound
\begin{align*}
\sum_{d \leq D} \tau(d)^A \frac{\alpha(d)}{d} 
\ll \int_\R \prod_{p \leq R}\left(1+ \sum_{j=1}^\infty \frac{(1+j)^A (\exp(O(a_{t,p^j}))-1)}{p^j}\right) |f(t)|\ dt.
\end{align*}
For $j \geq 2$ we use the crude bound
$$ \exp( O(a_{t,p^j} )) \ll j^{O(1)} p^{O(j/\log R)}\ll e^{O(j)}$$
for $p\leq R$ to conclude that
$$ \sum_{j=2}^\infty \frac{(1+j)^A (\exp(O(a_{t,p^j}))-1)}{p^j} \ll_A \frac{1}{p^2}$$
for any $p \leq R$.  For $j=1$  we have
$$ \exp( O(a_{t,p^j}) ) \leq 1+O(\min((1+|t|)\log_R p,1)),$$
and thus using $1+x\leq e^x$ we obtain
\begin{align*} 1+\sum_{j=1}^\infty \frac{(1+j)^A (\exp(O(a_{t,p^j}))-1)}{p^j} &\leq 1+\frac{O_A(\min((1+|t|)\log_R p,1))}{p}+O_A\left(\frac{1}{p^2}\right)\\
&\leq \exp\left( O_A\left( \frac{\min((1+|t|)\log_R p,1)}{p} + \frac{1}{p^2} \right) \right).
\end{align*}
From Mertens' theorem~\eqref{mertens} we have
\begin{align*}\prod_{p\leq R}\exp\left(  O_A( \frac{\min((1+|t|)\log_R p,1))}{p} + \frac{1}{p^2} \right)&=\exp\left(O_A\left( \sum_{p \leq R}\left(\frac{\min((1+|t|)\log_R p,1)}{p} +\frac{1}{p^2}\right)\right)\right)\\
&\ll (2+|t|)^{O_A(1)},
\end{align*}
and hence
$$\sum_{d \leq D} \tau(d)^A \frac{\alpha(d)}{d} 
\ll_A \int_\R  (2+|t|)^{O_A(1)}|f(t)|\ dt.$$
Using~\eqref{varphi} we obtain~\eqref{dao} as required.
\end{proof}
	
Now we can prove

\begin{proposition}[Replacing $\lambda_\Siegel$ with a Type I approximant]\label{step-3}  We have
\begin{align*}
& \E_{n \leq x} \Lambda_\Siegel(n+h_1) \cdots \Lambda_\Siegel(n+h_k) \lambda_\Siegel(n+h'_1) \cdots \lambda_\Siegel(n+h'_\ell) \\
&\quad \approx \E_{n \leq x} \Lambda_\Siegel(n+h_1) \cdots \Lambda_\Siegel(n+h_k) \lambda_\Siegel^\sharp(n+h'_1) \cdots \lambda_\Siegel^\sharp(n+h'_\ell).
\end{align*}
\end{proposition}

From Lemma~\ref{lamlam} we have
$$ \lambda_\Siegel^\sharp(n+h'_j) = \lambda_\Siegel(n+h'_j) + O(H(n+h'_j))$$
for $j=1,\dots,\ell$.  Multiplying these estimates using~\eqref{ls-bound} and the triangle inequality, and relabeling, we reduce to showing that
$$ \E_{n \leq x} |\Lambda_\Siegel(n+h_1) \cdots \Lambda_\Siegel(n+h_k)| H(n+h'_1) \cdots H(n+h'_{\ell'}) \approx 0$$
for any $1 \leq \ell' \leq \ell$.  By~\eqref{tri} it suffices to show that
\begin{equation}\label{lang}
 (\log^k x) \E_{n \leq x} \tau \nu(n+h_1) \cdots \tau \nu(n+h_k) H(n+h'_1) \cdots H(n+h'_{\ell'}) \approx 0.
\end{equation}
Expanding out~\eqref{geo-def}, the left-hand side is
$$
 \sum_{d'_{1},\dots,d'_{\ell'} \leq D} \alpha(d'_1) \cdots \alpha(d'_{\ell'}) (\log^k x)
\E_{n \leq x} \tau \nu(n+h_1) \cdots \tau \nu(n+h_k) 1_{d'_1|n+h'_1,\dots,d'_{\ell'}|n+h'_{\ell'}};
$$
using~\eqref{landreau-special}, one can bound this further by
$$
\ll \sum_{d_{1},\dots,d_{k'} \leq D} \tau(d_1)^{O(1)} \cdots \tau(d_k)^{O(1)} \alpha(d_{k+1}) \cdots \alpha(d_{k'}) (\log^k x)
\E_{n \leq x} \prod_{j=1}^k\nu(n+h_j) \prod_{j=1}^{k'} 1_{d_i|n+h_i}
$$
where we use the notations $k' \coloneqq k + \ell'$, $h_{k+j} \coloneqq h'_j$, and $d_{k+j} \coloneqq d'_j$ for $j=1,\dots,\ell'$.
Applying Lemma~\ref{nusieve}, we can bound this by
\begin{align*}
& \ll_A (\log_R^k x) \int_1^\infty \sum_{d_{1},\dots,d_{k''} \leq D} \frac{\tau(d_1)^{O(1)} \cdots \tau(d_{k'})^{O(1)} \alpha(d_{k+1}) \cdots \alpha(d_{k'})}{d_1 \cdots d_{k'}}\\ 
&\quad \quad \left(\prod_{ p|d_1 \cdots d_k} \min(\sigma \log_R p,1)\right) \frac{d\sigma}{\sigma^A}+ \frac{D^{k'} R^{2k}}{x}
\end{align*}
for any $A>0$.  The contribution of the latter term $\frac{D^{k'} R^{2k}}{x}$ is $\approx 0$ thanks to~\eqref{rdk0}.  By~\eqref{dao}, the former term can be bounded by
$$ \ll_A \frac{\log_R^k x}{\log^A \eta} \int_1^\infty \sum_{d_{1},\dots,d_{k} \leq D} \frac{\tau(d_1)^{O(1)} \cdots \tau(d_{k})^{O(1)}}{d_1 \cdots d_{k}} \left(\prod_{ p|d_1 \cdots d_k} \min(\sigma \log_R p,1)\right) \frac{d\sigma}{\sigma^A}$$
for any $A>0$, which by~\eqref{euler-stop} can be bounded by
$$ \ll_A \frac{\log_R^k x}{\log^A \eta} \int_1^\infty \prod_{p \leq D} E_p(\sigma) \frac{d\sigma}{\sigma^A}$$
where
$$ E_p(\sigma) \coloneqq \sum_{d_{1},\dots,d_{k} \in \N_{(p)}} \frac{\tau(d_1)^{O(1)} \cdots \tau(d_{k})^{O(1)}}{d_1 \cdots d_k} 
\min(\sigma \log_R p,1)^{1_{p|d_1 \cdots d_k}}.$$
We can bound
$$ E_p(\sigma) \leq 1 + O\left( \frac{\min(\sigma \log_R p, 1)}{p} \right)$$
so by~\eqref{mertens-alt}, \eqref{rxr1} and setting $A$ large enough we conclude~\eqref{lang}.  This completes the proof of Proposition~\ref{step-3}.

\section{Fourth step: replacing the von Mangoldt Siegel model with a Type I approximant}\label{sec:typei-vonmangoldt}

We now execute step (iv) of the strategy outlined in the introduction. In this step we will achieve power savings in many of our error terms, and as a consequence we can often afford to lose factors such as $x^{O(\eps)}$, in contrast to other sections where even a loss of $\log x$ is often unacceptable.

It is convenient to perform a smooth dyadic decomposition of the convolution $\chi * \log$ in order to run a smoothed version of the Dirichlet hyperbola method.  Let $\phi \colon \R \to \R$ be a smooth even function supported on $[-1,1]$ of total mass one.  For any $t>0$, define the function
$$ \Phi_t(n) \coloneqq \phi\left(\log \frac{n}{t}\right),$$
which is a smooth cutoff to the interval $[t/e, et]$. Then for any natural number $n$, one has the identity
\begin{equation}\label{log-ident}
\begin{split} \int_0^\infty \Phi_t(n) \log t \frac{dt}{t} &= 
\int_0^\infty \phi\left(\log\frac{n}{t}\right) \log t \frac{dt}{t} \\
&= \int_\R \phi(u) (\log n - u)\ du \\
&= \log n,
\end{split}
\end{equation}
where we made the change of variables $u \coloneqq \log n - \log t$.  We conclude that
\begin{equation}\label{leo}
\chi * \log(n) = \int_0^\infty \chi*\Phi_t(n) \log t \frac{dt}{t}.
\end{equation}

As it turns out, the Dirichlet convolution $\chi * \Phi_t$ is of an adequate ``Type I'' form when $t \leq Dq^2_\chi$ or $x/t \leq (Dq^2_\chi)^2$.  Accordingly, we split
$$ \chi*\log =  (\chi*\log)^\sharp + (\chi*\log)^\flat$$
where $(\chi*\log)^\sharp$ is the Type I approximant
\begin{equation}\label{chilog}
\begin{split}
 (\chi*\log)^\sharp(n) &= \int_0^{Dq^2_\chi} \chi * \Phi_t(n) \log t \frac{dt}{t} \\
&\quad\quad  + \int_{Dq^2_\chi}^\infty \psi_{\leq (Dq^2_\chi)^2}(x/t) \chi * \Phi_t(n) \log t \frac{dt}{t} \\
&\quad\quad  + \int_{Dq^2_\chi}^\infty \psi_{> (Dq^2_\chi)^2}(x/t) \chi * \Phi_{Dq^2_\chi}(n) \log t \frac{dt}{t},
\end{split}
\end{equation}
and $(\chi*\log)^\flat$ is the error
\begin{equation}\label{chilog-flat}
(\chi*\log)^\flat(n) \coloneqq \int_{Dq^2_\chi}^\infty \psi_{>(Dq^2_\chi)^2}(x/t) \chi * (\Phi_t - \Phi_{Dq^2_\chi}) \log t \frac{dt}{t}.
\end{equation}
Thus, $(\chi*\log)^\sharp$ is the modification of $\chi*\log$ formed by replacing the cutoff $\Phi_t$ with $\Phi_{Dq^2_\chi}$ in the intermediate range $Dq^2_\chi \leq t \leq \frac{x}{(Dq^2_\chi)^2}$ of $t$ (using a smoothed version of the upper cutoff $t \leq \frac{x}{(Dq^2_\chi)^2}$ in order to facilitate some technical computations in the next section).   As it turns out, it will be the second term in the right-hand side of~\eqref{chilog} (in which the $\Phi_t$ term is supported in values $\gg x/(Dq^2_\chi)^2$, so that the $\chi$ term is supported in values $\ll (D q^2_\chi)^2$) that will give the main contributions, being a more complicated version of the (untwisted) Type I sum $(\chi \psi_{\leq (Dq^2_\chi)^2}) * \log$.

We then have a similar spliting
$$ \Lambda_\Siegel = \Lambda_\Siegel^\sharp + \Lambda_\Siegel^\flat$$
where
$$\Lambda_\Siegel^\sharp \coloneqq (\chi*\log)^\sharp \nu$$
and
$$\Lambda_\Siegel^\flat \coloneqq (\chi*\log)^\flat \nu.$$

We have good bounds on the distribution of $(\chi*\log)^\flat$ or $\Lambda_\Siegel^\flat$ in residue classes $a\ (q)$ with $q$ almost as large as $x^{2/3}$, as long as $(a,q)$ is not too large:

\begin{proposition}[$2/3$ level of distribution]\label{two-thirds}  Let $0 < \eps < \frac{1}{2}$, $1 \leq q \leq x$, and $a$ be an integer. Let $I$ be a subinterval of $[0,2x]$.  Let $f\colon \Z \to [-1,1]$ be a $q_\chi$-periodic function.
\begin{itemize}
\item[(i)]  We have
$$ \sum_{n \in I: n = a\ (q)} (\chi*\log)^\flat(n) f\left(\frac{n-a}{q} \right) \ll_\eps \frac{x}{q} \left( \frac{(a,q)^{3/2} q_\chi^{9/2} q^{3/2}}{x^{1-O(\eps)}} + \frac{x^{O(\eps)} (a,q)^2}{D^{1/2}} + \frac{1}{x^{\eps}} \right).$$
\item[(ii)] If $\eps$ is sufficiently small depending on $k,\ell,\eps_0$, then
we have
$$ \sum_{n \in I: n = a\ (q)} \Lambda_\Siegel^\flat(n) f\left(\frac{n-a}{q} \right) \ll_\eps (a,q)^{O(1)} \frac{x}{q} \left( \frac{q_\chi^{9/2} q^{3/2}}{x^{1-\eps_0}} + \frac{1}{x^{\eps}} \right).$$
\end{itemize}
\end{proposition}

The powers of $(a,q)$ and $q_\chi$ are of minor importance and these terms can be neglected on a first reading.  The key point here is that we can have a power savings over the trivial bound of $O_\eps(x^{1+\eps}/q)$ even when $q$ is somewhat above $x^{1/2}$ (indeed, the above bounds can remain non-trivial as $q$ approaches $x^{2/3}$).

\begin{proof}  We first prove (i).  Note that $0\leq (\chi*\log)^{\flat}(n)\leq \chi*\log(n)$. From~\eqref{divisor-bound} we may bound the left-hand side of the claim by $O_\eps(x^{O(\eps)} ( 1 + x/q ))$.  From this we see that we may assume without loss of generality that we may take $\eps$ is sufficiently small depending on $k'$, and we may also assume that $q \leq x^{2/3}$, since otherwise the above crude bound is already dominated by $x^{O(\eps)} q^{1/2}$ and hence by $\frac{x}{q} \frac{(a,q)^{3/2} q_\chi^{9/2} q^{3/2}}{x^{1-O(\eps)}}$.  By shrinking $I$ slightly (and using~\eqref{divisor-bound} to treat the error) we may assume that $I \subset [x^{1-\eps}, 2x]$.

The integrand in~\eqref{chilog-flat} is only non-zero in the range $Dq^2_\chi \leq t \leq x /(Dq^2_\chi)^2$.
By the fundamental theorem of calculus one has
$$ \Phi_t(n) - \Phi_{Dq^2_\chi}(n) = -\int_{Dq^2_\chi}^t \tilde \Phi_{t'}(n)\frac{dt'}{t'}$$
where
$$ \tilde \Phi_t(n) \coloneqq \phi'\left(\log \frac{n}{t}\right)$$ 
so by the triangle inequality (and increasing $\eps$ slightly) it will suffice to show that
$$ \sum_{n = a\ (q)} f\left(\frac{n-a}{q} \right) 1_I(n) \chi * \tilde \Phi_t(n) \ll_\eps x^{O(\eps)} (a,q)^{3/2} q_\chi^{9/2} q^{1/2} + \frac{x^{1+O(\eps)} (a,q)^2}{D^{1/2} q} + \frac{x^{1-\eps}}{q}$$
for all $Dq^2_\chi \leq t \leq x/(Dq^2_\chi)^2$.

We can approximate $1_I$ by a cutoff $\psi_I \colon \R \to \R$ supported on $I$ obeying $\psi_I(y)=1$ whenever $\textnormal{dist}(y,I)\geq x^{1-2\varepsilon}$, and additionally obeying the derivative estimates 
\begin{equation}\label{psijy}
\psi_I^{(j)}(y) \ll_j x^{-(1-2\eps) j}
\end{equation}
for all $j \geq 0$ and $y \in \R$, with the error being acceptable by~\eqref{divisor-bound}.  It thus remains to establish the bound
$$ \sum_{n = a\ (q)} f\left(\frac{n-a}{q} \right) \psi_I(n) \chi * \tilde \Phi_t(n) \ll_\eps x^{O(\eps)} (a,q)^{3/2} q_\chi^{9/2} q^{1/2} + \frac{x^{1+O(\eps)} (a,q)^2}{D^{1/2} q} .$$
The left-hand side can be rewritten as
$$ \sum_{n_1,n_2} \chi(n_1) \tilde \Phi_t(n_2) \psi_I(n_1 n_2) 1_{n_1 n_2 = a\ (q)} f\left(\frac{n_1 n_2-a}{q}\right).$$
By the triangle inequality, it suffices to show that
$$ \sum_{n_1,n_2} \chi(n_1) \tilde \Phi_t(n_2) \psi_I(n_1 n_2) 1_{n_1 n_2 = a'\ (qq_\chi)} \ll_\eps x^{O(\eps)} (a,q)^{3/2} q_\chi^{7/2} q^{1/2} + \frac{x^{1+O(\eps)} (a,q)^2}{D^{1/2} q q_\chi}.$$
for any $a' \in \Z/(qq_\chi\Z)$ with $a' = a\ (q)$.  If we set $q_0 \coloneqq (a,q) q_\chi$, then $(a', qq_\chi)$ divides $q_0$.  Writing
$$ q'_0 \coloneqq ((a',qq_\chi)q_0, qq_\chi)$$
and using Lemma~\ref{mfe} and the triangle inequality and~\eqref{divisor-bound}, we obtain the bound
\begin{equation}\label{insert}
\sum_{n_1,n_2} 
\chi(n_1) \tilde \Phi_t(n_2) \psi_I(n_1 n_2) 1_{n_1 n_2 = a'\ (qq_\chi)}  \ll_\eps x^{\eps} \left( \frac{q'_0}{qq_\chi} X + q_0^{3/2} q^{-3/2} Y \right)
\end{equation}
where
$$ X \coloneqq \left| \sum_{n_1,n_2} \chi(n_1) \tilde \Phi_t(n_2) \psi_I(n_1 n_2) 1_{n_1 n_2=a'\ (q'_0)} 
1_{(n_1 n_2,qq_\chi)=(a',qq_\chi)}\right|$$
and
$$ Y \coloneqq \sum_{u_1,u_2 \in \Z/(qq_\chi\Z): qq_\chi/q_0 \nmid u_1, u_2} (u_1,u_2,qq_\chi)^{1/2}
\left| \sum_{n_1,n_2} \tilde \Phi_t(n_2) \psi_I(n_1 n_2) e_{qq_\chi}(u_1 n_1 + u_2 n_2) \right|.$$
We first estimate the quantity $Y$.  From repeated summation by parts we have
\begin{align*}
\sum_{n_1,n_2} \tilde \Phi_t(n_2) \psi_I(n_1 n_2) e_q(u_1 n_1 + u_2 n_2) &\ll_\eps x^{-1+O(\eps)} \frac{x/t}{\|u_1/(qq_\chi)\|_{\R/\Z}} \frac{t}{\|u_2/(qq_\chi)\|_{\R/\Z}} \\
&= \frac{x^{O(\eps)}}{\|u_1/qq_\chi\|_{\R/\Z} \|u_2/qq_\chi\|_{\R/\Z}}.
\end{align*}
Writing $u_1 = d u'_1$, $u_2 = d u'_2$ with $d = (u_1,u_2,qq_\chi)^{1/2}$, we then have
\begin{align*}
 Y &\ll_\eps \sum_{d|qq_\chi} d^{1/2} \sum_{1 \leq u'_1, u'_2 < \frac{qq_\chi}{d}}
\frac{x^{O(\eps)}}{\|u'_1/(qq_\chi/d)\|_{\R/\Z} \|u'_2/(qq_\chi/d)\|_{\R/\Z}} \\
&\ll_\eps x^{O(\eps)} \sum_{d|qq_\chi} d^{1/2} \left(\frac{qq_\chi}{d} \log\left(2+\frac{qq_\chi}{d}\right)\right)^2 \\
&\ll_\eps x^{O(\eps)} q^2 q_\chi^2.
\end{align*}
Thus we see that the contribution of $Y$ is acceptable.  Now we consider the contribution of $X$.  From M\"obius inversion we have
$$1_{(n_1 n_2,qq_\chi)=(a',qq_\chi)} = \sum_{d: (a',qq_\chi) | d | qq_\chi} \mu\left(\frac{d}{(a',qq_\chi)}\right) 1_{d|n_1 n_2}.$$
By~\eqref{divisor-bound} and the triangle inequality, one thus has
$$ X \ll_\eps x^{O(\eps)} \sum_{n_1,n_2} \chi(n_1) \tilde \Phi_t(n_2) \psi_I(n_1 n_2) 1_{n_1 n_2=a'\ (q'_0)} 
1_{d|n_1 n_2}$$
for some $d$ with $(a',qq_\chi) | d | qq_\chi$.  On the one hand, we see from~\eqref{divisor-bound} (noting that the constraints $n_1n_2 = a'\ (q'_0)$, $d|n_1 n_2$ constrain $n_1 n_2$ to at most one residue class modulo $[d,q'_0]$) that
\begin{equation}\label{xd}
 X \ll_\eps x^{O(\eps)} \frac{x}{[d,q'_0]}.
\end{equation}
On the other hand, we can write
$$ X \ll_\eps x^{O(\eps)} \left| \sum_{n_1,n_2} F(n_1,n_2) \phi'\left(\log \frac{n_2}{t}\right) \psi_I(n_1 n_2) \right| $$
where $F$ is the $[d,q'_0]$-periodic function
$$ F(n_1,n_2) \coloneqq \chi(n_1) 1_{n_1 n_2 = a'\ (q'_0)} 1_{d|n_1 n_2}.$$
By Fourier expansion and Poisson summation, we can then write
$$ X \ll_\eps x^{O(\eps)} \left| \sum_{\xi_1, \xi_2 \in \frac{\Z}{dq'_0}} \hat F(\xi_1,\xi_2) \Psi(\xi_1,\xi_2) \right|$$
where
$$ \hat F(\xi_1,\xi_2) \coloneqq \E_{n_1,n_2} F(n_1,n_2) e(-n_1 \xi_1 - n_2 \xi_2)$$
and
$$ \Psi(\xi_1,\xi_2) \coloneqq \int_0^\infty \int_0^\infty \tilde \Phi_t(x_2) \psi_I(x_1 x_2) e(x_1 \xi_1 + x_2 \xi_2)\ dx_1 dx_2.$$
From the area-preserving change of variables $(u_1,u_2) \coloneqq (\log \frac{x_2}{t}, x_1 x_2)$ and the fundamental theorem of calculus we have
$$ \Psi(0,0) = \int_\R \int_0^\infty \phi'(u_1) \psi_I(u_2)\ du_1 du_2 = 0,$$
and from integration by parts one has the bounds
$$ \Psi(\xi_1,\xi_2) \ll_m \frac{x^{1+O_m(\eps)}}{(1 + t|\xi_1|)^{m} (1 + \frac{x}{t}|\xi_2|)^{m}}$$
for any $m \geq 0$.  Meanwhile, using the trivial bound $|\hat F(\xi_1,\xi_2)| \leq 1$
and $t,\frac{x}{t} \geq D q^2_\chi$, we have
$$ X \ll_{m,\eps} x^{O_m(\eps)} \sum_{\xi \in (\frac{\Z}{[d,q'_0]})^2 \backslash \{(0,0)\}} (1 + D q^2_\chi |\xi|)^{-m}$$
for any $m>0$, and thus
$$ X \ll_{m,\eps} x^{O_m(\eps)} \left(\frac{D q_\chi}{[d,q'_0]}\right)^{-m}$$
for any $m>0$.  In particular, if $[d,q'_0] \leq D^{1/2} q^2_\chi$, we have $X \ll_\eps 1$ (say) by choosing $m$ large enough.  Comparing this with~\eqref{xd}, we conclude that
$$ X \ll_\eps x^{O(\eps)} \frac{x}{D^{1/2} q^2_\chi}$$
in all cases.  Inserting these bounds back into~\eqref{insert} and writing $q_0 = (a,q) q_\chi$ and bounding 
$$q'_0 \leq (a',qq_\chi) q_0 \leq q_0^2 = (a,q)^2 q_\chi^2,$$
we conclude that
$$ \sum_{n_1,n_2} \chi(n_1) \tilde \Phi_t(n_2) \psi_I(n_1 n_2) 1_{n_1 n_2 = a'\ (qq_\chi)}  \ll_\eps x^{O(\eps)} \left( \frac{x (a,q)^2}{D^{1/2} q q_\chi} + (a,q)^{3/2} q_\chi^{7/2} q^{1/2} \right),$$
and the claim (i) follows.

Now we prove (ii).  Expanding out the Selberg sieve $\nu$ as
\begin{equation}\label{nid}
 \nu(n) = \sum_{d \leq R^2} a_d 1_{d|n}
\end{equation}
for some sieve weights $a_d$ that can be crudely bounded using~\eqref{divisor-bound} as
\begin{equation}\label{nu-bound}
a_d \ll \tau(d) \ll_\eps x^\eps,
\end{equation}
we see from the triangle inequality and pigeonhole principle (noting that $\sum_{d \leq R^2} \frac{1}{d} \ll_\eps x^\eps$) that
$$ \sum_{n \in I: n = a\ (q)} \Lambda_\Siegel^\flat(n) f\left(\frac{n-a}{q} \right) \ll_\eps 
x^{2\eps} \left| d \sum_{n \in I: n = a\ (q); d|n} (\chi*\log)^\flat(n) f\left(\frac{n-a}{q} \right) \right|$$
for some $d \leq R^2$.  We can restrict attention to those $d$ with $(d,q) | (a,q)$, since otherwise the sum is empty.  The conditions $n = a\ (q)$, $d|n$ can then be combined into a single congruence class $n = a'\ ([q,d])$, with $(a',[q,d]) \leq d (a,q)$; on this class, the quantity $f\left(\frac{n-a}{q} \right)$ can be viewed as a $q_\chi$-periodic function of $\frac{n-a'}{[q,d]}$.  Applying (i) (with $\eps$ replaced by $3\eps$) we have
\begin{align*} &d \sum_{n \in I: n = a\ (q); d|n} (\chi*\log)^\flat(n) f\left(\frac{n-a}{q} \right)\\
&\quad \quad \ll 
\frac{dx}{[q,d]} \left( \frac{d^{3/2} (a,q)^{3/2} q_\chi^{9/2} [q,d]^{3/2}}{x^{1-O(\eps)}} + \frac{x^{O(\eps)} d^{3/2} (a,q)^2}{D^{1/2}} + \frac{1}{x^{3\eps}} \right).
\end{align*}
Writing $\frac{d}{[q,d]} = \frac{(d,q)}{q} \leq \frac{(a,q)}{q}$ and then bounding $[q,d] \leq qd$ and $d \leq R^2$, we conclude
$$ \sum_{n \in I: n = a\ (q)} \Lambda_\Siegel^\flat(n) f\left(\frac{n-a}{q} \right) \ll_\eps 
\frac{x}{q}
\left( \frac{R^6 (a,q)^{5/2} q_\chi^{9/2} q^{3/2}}{x^{1-O(\eps)}} + \frac{x^{O(\eps)} R^{3} (a,q)^3}{D^{1/2}} + \frac{(a,q)}{x^{\eps}} \right).$$
Using~\eqref{R0-def}, \eqref{D-def}, we obtain the claim for $\eps$ small enough.
\end{proof}

Now we can prove

\begin{proposition}[Replacing $\Lambda_\Siegel$ with a Type I approximant]\label{step-4} Assume $k \leq 2$. Then we have
\begin{align*}
& \E_{n \leq x} \Lambda_\Siegel(n+h_1) \cdots \Lambda_\Siegel(n+h_k) \lambda_\Siegel^\sharp(n+h'_1) \cdots \lambda_\Siegel^\sharp(n+h'_\ell) \\
&\quad
\approx \E_{n \leq x} \Lambda_\Siegel^\sharp(n+h_1) \cdots \Lambda_\Siegel^\sharp(n+h_k) \lambda_\Siegel^\sharp(n+h'_1) \cdots \lambda_\Siegel^\sharp(n+h'_\ell).
\end{align*}
\end{proposition}

\begin{proof}  The claim is trivial for $k=0$, so we may assume that $k\in \{1,2\}$.  
By the triangle inequality and relabeling it suffices to show the bounds
\begin{equation}\label{lb-1}
\E_{n \leq x} \Lambda_\Siegel^\flat(n+h_1) \lambda_\Siegel^\sharp(n+h'_1) \cdots \lambda_\Siegel^\sharp(n+h'_\ell) \approx 0
\end{equation}
when $k=1$, and the bounds
\begin{equation}\label{lb-2}
\E_{n \leq x} \Lambda_\Siegel^\flat(n+h_1) \Lambda_\Siegel(n+h_2) \lambda_\Siegel^\sharp(n+h'_1) \cdots \lambda_\Siegel^\sharp(n+h'_\ell) \approx 0
\end{equation}
and
\begin{equation}\label{lb-3}
\E_{n \leq x} \Lambda_\Siegel^\flat(n+h_1) \Lambda_\Siegel^\sharp(n+h_2) \lambda_\Siegel^\sharp(n+h'_1) \cdots \lambda_\Siegel^\sharp(n+h'_\ell) \approx 0
\end{equation}
when $k=2$.

We begin with~\eqref{lb-1}. Let $\eps>0$ be a sufficiently small quantity. From~\eqref{ls-def} and~\eqref{divisor-bound} we have
\begin{equation}\label{lss}
 \lambda_\Siegel^\sharp(n) = \sum_{d \leq D} b_d 1_{d|n} \chi(n/d)
\end{equation}
for some weights $b_d$ of size 
\begin{equation}\label{bd-bound}
b_d \ll \tau(d) \log^{O(1)} x \ll_\eps x^\eps.
\end{equation}
Since
$$ \sum_{d'_1,\dots,d'_{\ell} \leq D} \frac{1}{[d'_1,\dots,d'_{\ell}]} \leq \sum_{d \leq D^{k'-1}} \frac{\tau(d)^{k'-1}}{d} \ll_\eps x^\eps$$
(thanks to~\eqref{divisor-bound}), we thus see from the pigeonhole principle that the left-hand side of~\eqref{lb-1} is bounded by
$$
\ll_\eps x^{O(\eps)} [d'_1,\dots,d'_\ell] \left| \E_{n \leq x} \Lambda_\Siegel^\flat(n+h_1) \prod_{j=1}^{\ell} 1_{d'_j|n+h'_j} \chi\left(\frac{n+h'_j}{d'_j}\right) \right|$$
for some $d'_1,\dots,d'_{\ell} \leq D$.  By translating (and removing negligible errors) we may assume that $h_1=0$.  Setting $d \coloneqq [d'_1,\dots,d'_{\ell}]$, we see that the constraints $d'_j|n+h'_j$ are either inconsistent, or restrict $n$ to a single residue class $a\ (d)$ with $(a,d) \ll 1$, and then $\prod_{j=1}^{\ell} \chi(\frac{n+h'_j}{d'_j})$ is a $q_\chi$-periodic function of $\frac{n-a}{d}$. Applying Proposition~\ref{two-thirds}(ii) (with a suitable multiple of $\eps$), we bound the left-hand side of~\eqref{lb-1} by
\begin{equation}\label{deps}
 \ll_\eps \frac{x^{O(\eps)} q_\chi^{9/2} d^{3/2}}{x^{1-\eps_0}} + \frac{1}{x^{\eps}}.
\end{equation}
Bounding $d \leq D^{k'-1}$ and using~\eqref{dsmash}, we see that the right-hand side is $\approx 0$ for $\eps$ a sufficiently small constant, giving the claim.

Now we consider~\eqref{lb-2}, \eqref{lb-3}.  Again let $\eps>0$ be sufficiently small. From~\eqref{leo}, \eqref{chilog} and the pigeonhole principle we can bound both left-hand sides (up to negligible errors) by
\begin{align}\label{bo}\begin{split}
 &\ll_\eps x^{O(\eps)} \Bigg| \E_{n \leq x} \Lambda_\Siegel^\flat(n+h_1) \left(\sum_{d_2|n+h_2} \chi(d_2) \Phi_{t}\left(\frac{n+h_2}{d_2} \right)\right) \nu(n+h_2)\\
 &\quad \quad \lambda_\Siegel^\sharp(n+h'_1) \cdots \lambda_\Siegel^\sharp(n+h'_{\ell}) \Bigg|
 \end{split}
\end{align}
for some $1 \ll t \ll x$ (note that the summation vanishes for $t$ outside this range).

We now use a version of the Dirichlet hyperbola method.  First suppose that $t \geq \sqrt{x}$, then the summand vanishes unless $d_2 \ll \sqrt{x}$.  Expanding out using~\eqref{lss}, \eqref{nid} much as before and now using
$$ \sum_{d_2 \ll \sqrt{x}; \tilde d_2 \leq R^2; d'_1,\dots,d'_\ell \leq D} \frac{1}{[d_2,\tilde d_2,d'_1,\dots,d'_{\ell}]} \ll_\eps x^\eps$$
we can bound  the contribution of the $d_2 \leq \sqrt{x}$ case by
$$ \ll_\eps x^{O(\eps)} [d_2,\tilde d_2,d'_1,\dots,d'_\ell] 
\left| \E_{n \leq x} \Lambda_\Siegel^\flat(n+h_1) 1_{d_2,\tilde d_2|n+h_2} \Phi_{t}\left(\frac{n+h_2}{d_2}\right) \prod_{j=1}^{\ell} 1_{d'_j|n+h'_j} \chi\left(\frac{n+h'_j}{d'_j}\right) \right|$$
for some $d_2 \leq \sqrt{x}$, $\tilde d_2 \leq R^2$, and $d'_1,\dots,d'_\ell \leq D$.
Writing $d \coloneqq [d_2,\tilde d_2,d'_1,\dots,d'_{\ell}]$ and arguing as before, using summation by parts to deal with the slowly varying function $\Phi_{t}(\frac{n+h_2}{d_2})$, we can again bound this expression by~\eqref{deps}.  Bounding 
$d \ll \sqrt{x} R^2 D^{k'-2}$ and using~\eqref{dsmash}, we see that the right-hand side is $\approx 0$ for $\eps$ small enough, giving the claim.

Finally, suppose that $t < \sqrt{x}$.  Now we make the change of variables $d^*_2 \coloneqq \frac{n+h_2}{d_2}$ and rewrite the bound as
\begin{align*}&\ll_\eps x^{O(\eps)} \Bigg| \E_{n \leq x} \Lambda_\Siegel^\flat(n+h_1) \left(\sum_{d^*_2|n+h_2} \Phi_t(d^*_2) \chi(\frac{n+h_2}{d^*_2})\right) \nu(n+h_2)\\
&\quad \quad \lambda_\Siegel^\sharp(n+h'_1) \cdots \lambda_\Siegel^\sharp(n+h'_{\ell}) \Bigg|.
\end{align*}
Observe that the summand vanishes unless $d^*_2 \ll \sqrt{x}$.  Now we can repeat the previous arguments (using $d_2^*$ in place of $d_2$, and the $q_\chi$-periodic function $\chi$ in place of $\Phi_t$, noting that~\eqref{dsmash} can handle several additional losses of $q_\chi$) to conclude.
\end{proof}

\section{Fifth step: Computing the Type I correlations}\label{sec:mainterm}

We now execute step (v) of the strategy outlined in the introduction by establishing

\begin{proposition}[Evaluating the Type I correlation]\label{step-5} 
We have
\begin{equation}\label{enx}
\E_{n \leq x} \Lambda_\Siegel^\sharp(n+h_1) \cdots \Lambda_\Siegel^\sharp(n+h_k) \lambda_\Siegel^\sharp(n+h'_1) \cdots \lambda_\Siegel^\sharp(n+h'_\ell) \approx {\mathfrak S}
\end{equation}
where ${\mathfrak S}$ is the quantity in Conjecture~\ref{hlc}.
\end{proposition}

Clearly Theorem~\ref{main} follows immediately from concatenating together Propositions~\ref{step-1}, \ref{step-2}, \ref{step-3}, \ref{step-4}, \ref{step-5} using~\eqref{chain}.

We first dispose of the easy case $\ell > 0$, in which ${\mathfrak S}$ vanishes.  For $1 \leq j \leq k$, we see from~\eqref{chilog} and replacing $d$ by $n/d$ in the first and third factors, and truncating the very small or very large values of $t$ (where the summand vanishes) that
\begin{equation}\label{chilog-alt}
\begin{split}
 (\chi*\log)^\sharp(n) &= \int_{1/100}^{Dq^2_\chi}\sum_{d|n} \Phi_t(d)\chi(n/d) \log t \frac{dt}{t} \\
&\quad + \int_{Dq^2_\chi}^{100x} \psi_{\leq (Dq^2_\chi)^2}(x/t) \sum_{d\mid n}\Phi_t(n/d)\chi(d) \log t \frac{dt}{t} \\
&\quad + \int_{Dq^2_\chi}^{100x} \psi_{> (Dq^2_\chi)^2}(x/t) \sum_{d\mid n}\Phi_{Dq^2_\chi}(d)\chi(n/d) \log t \frac{dt}{t}.
\end{split}
\end{equation}
In all of these terms, the summands vanish unless $d \ll (D q^2_\chi)^2$.  One can then write
$$  (\chi*\log)^\sharp(n) = \sum_{d \ll (D q^2_\chi)^2: d|n}( \Psi(n/d) \chi(d) + c_d \chi(\frac{n}{d})),$$
where $\Psi \colon \R^+ \to \R$ is the smooth function
\begin{equation}\label{naq}
 \Psi(y) \coloneqq \int_{Dq^2_\chi}^{100x} \psi_{\leq (Dq^2_\chi)^2}(x/t)  \Phi_t(y) \log t \frac{dt}{t}
\end{equation}
and $c_d$ is the coefficient
\begin{align*}
c_d &\coloneqq \int_{1/100}^{Dq^2_\chi}  \Phi_t(d) \log t \frac{dt}{t} \\
&\quad + \int_{Dq^2_\chi}^{100x} \psi_{> (Dq^2_\chi)^2}(x/t) \Phi_{Dq^2_\chi}(d) \log t \frac{dt}{t}.
\end{align*}
For the current analysis we will need the crude bound
$$ \| \Psi \|_{\mathrm{TV}} \ll \log^{O(1)} x, c_d\ll \log^{O(1)}x,$$
where we use the total variation norm
$$ \|f\|_{\mathrm{TV}} \coloneqq \sup_{y>0} |f(y)| + \int_\R |f'(y)|\ dy.$$

Combining this with the expansion~\eqref{nid}, we see that
\begin{equation}\label{lsn}
  \Lambda^\sharp(n) = \sum_{d \ll R^2 (D q^2_\chi)^2: d|n} \Psi_d(n) + \sum_{d \ll R^2 (D q^2_\chi)^2: d'|d|n} g_{d,d'} \chi(\frac{n}{d'})
	\end{equation}
where $\Psi_d \colon \R^+ \to \R$ is a smooth function and $g_{d,d'}$ is a coefficient obeying the bounds
\begin{equation}\label{psid-tv}
 \| \Psi_d \|_{\mathrm{TV}}\ll \tau(d)^{O(1)}, g_{d,d'} \ll \tau(d)^{O(1)} \log^{O(1)} x.
\end{equation}
Using the decomposition~\eqref{lss} to expand $\lambda_\Siegel^\sharp(n+h'_j)$, we can thus write $\Lambda_\Siegel^\sharp(n+h_1) \cdots \Lambda_\Siegel^\sharp(n+h_k) \lambda_\Siegel^\sharp(n+h'_1) \cdots \lambda_\Siegel^\sharp(n+h'_\ell)$ as
\begin{align*} &\sum_{J \subset \{1,\dots,k\}}\,\, \sum_{d_1,\dots,d_k \ll R^2 (D q^2_\chi)^2; d_{k+1},\dots,d_{k+\ell} \leq D; d'_j|d_j \forall j \in J}
h_{d_1,\dots,d_{k+\ell},d'_1,\dots,d'_k}(n)\\
&\quad \quad \left(\prod_{j=1}^{k+\ell} 1_{d|n+h_j}\right) \prod_{j \in J \cup \{k+1,\dots,k+\ell\}} \chi\left(\frac{n+h_j}{d'_j}\right)
\end{align*}
for some smooth functions $h_{d_1,\dots,d_{k+\ell},d'_1,\dots,d'_k} \colon \R \to \R$ with
$$ \| h_{d_1,\dots,d_{k+\ell},d'_1,\dots,d'_k}\|_{\mathrm{TV}} \ll \tau(d_1)^{O(1)} \cdots \tau(d_k)^{O(1)} \log^{O(1)} x,$$
and with the convention that $h_{k+j} = h'_j$ and $d'_{k+j} = d_{k+j}$ for $j=1,\dots,\ell$.  From Lemma~\ref{and}(ii) and summation by parts to deal with the 
$h_{d_1,\dots,d_{k+\ell},d'_1,\dots,d'_k}$ coefficients, we may thus bound the left-hand side of~\eqref{enx} by
\begin{align*}
&\ll_\eps q_\chi^{1/2+\eps} \log^{O(1)} x \sum_{J \subset \{1,\dots,k\}}\,\, 
\sum_{d_1,\dots,d_k \ll R^2 (D q^2_\chi)^2; d_{k+1},\dots,d_{k+\ell} \leq D; d'_j|d_j \forall j \in J}\\
&\quad \quad (d_1 \cdots d_k, q_\chi)^{1/2} \tau(d_1)^{O(1)} \cdots \tau(d_k)^{O(1)}\left( \frac{1}{q_\chi d_1 \cdots d_k} + \frac{1}{x} \right)
\end{align*}
which on evaluating the $d'_j$ sums, and then writing $d \coloneqq d_1 \cdots d_k$, can be bounded by
\begin{equation}\label{eqc}
 \ll_\eps q_\chi^{1/2+\eps} \log^{O(1)} x \sum_{d \ll D^{2(k+\ell)} (R q_\chi^2)^{2k}} (d,q_\chi)^{1/2} \tau(d)^{O(1)} \left( \frac{1}{q_\chi d} + \frac{1}{x} \right).
\end{equation}
From~\eqref{rdk0} we see that
$$ d \ll D^{2k(k+\ell)} (R q_\chi^2)^{2k} \ll \frac{x}{q^{1/2+\eps_0/2}_\chi} $$
so that
$$ \frac{1}{q_\chi d} + \frac{1}{x} \ll \frac{1}{q_\chi^{1/2+\eps_0/2} d}.$$
and then by~\eqref{euler-stop} we can bound~\eqref{eqc} by
\begin{align}\label{tausum}
\ll_\eps q_\chi^{-\frac{\eps_0}{2}+\eps} \log^{O(1)} x \prod_{p \leq x} \sum_{d \in \N_{(p)}} \frac{(d,q_\chi)^{1/2} \tau(d)^{O(1)}}{d}.
\end{align}
One can calculate
$$ \sum_{d \in \N_{(p)}} \frac{(d,q_\chi)^{1/2} \tau(d)^{O(1)}}{d} \leq 1 + O(\frac{1}{p})$$
when $p \nmid q_\chi$ and
$$ \sum_{d \in \N_{(p)}} \frac{(d,q_\chi)^{1/2} \tau(d)^{O(1)}}{d} \ll 1$$
otherwise, thus by~\eqref{mertens-4} the preceding expression~\eqref{tausum} is
$$
\ll_\eps q_\chi^{-\frac{\eps_0}{2}+\eps} \tau(q_\chi)^{O(1)} \log^{O(1)} x$$
which by~\eqref{divisor-bound} is
$$
\ll_\eps q_\chi^{-\frac{\eps_0}{4}} \log^{O(1)} x$$
if $\eps$ is small enough. Applying~\eqref{etao} we conclude that
$$ \E_{n \leq x} \Lambda_\Siegel^\sharp(n+h_1) \cdots \Lambda_\Siegel^\sharp(n+h_k) \lambda_\Siegel^\sharp(n+h'_1) \cdots \lambda_\Siegel^\sharp(n+h'_\ell)  \approx 0.$$
This concludes the treatment of the $\ell > 0$ case.

Now suppose that $\ell=0$.  The above arguments allow us to dispose of the $g_{d,d'}$ contributions in~\eqref{lsn}, leaving us with the task of showing that
$$ \E_{n \leq x} \prod_{j=1}^k \sum_{d_j \ll R^2 (D q^2_\chi)^2: d_j|n+h_j} \Psi_{d_j}(n+h_j) \approx {\mathfrak S}.$$
This is a correlation of Goldston--Y{\i}ld{\i}r{\i}m type and can be calculated by a lengthy but straightforward calculation, basically a more careful variant of Lemma~\ref{nusieve}.  We follow the Fourier-analytic method laid out in~\cite[Appendix D]{gt-linear}, as follows. Using Lemma~\ref{chin},~\eqref{psid-tv}, and summation by parts, we can write the left-hand side here as
\begin{align*}&\sum_{d_1,\dots,d_k \ll R^2 (D q^2_\chi)^2}\Bigg(
\frac{\prod_{1 \leq i < j \leq k} 1_{(d_i,d_j)|h_i-h_j}}{[d_1,\dots,d_k]} \frac{1}{x} \int_0^x \prod_{j=1}^k \Psi_{d_j}(y+h_j)\ dy\\
&\quad \quad + O\left( \frac{\tau(d_1)^{O(1)} \cdots \tau(d_k)^{O(1)} \log^{O(1)} x}{x} \right)\Bigg).
\end{align*}
Using~\eqref{divisor-bound}, the contribution of the error term is at most
$$ \ll_\eps \frac{(R D q^2_\chi)^{2k} x^\eps}{x}$$
for any $\eps>0$, which is $\approx 0$ for $\eps$ small enough thanks to~\eqref{rdk0}.  Thus it remains to show that
$$ \sum_{d_1,\dots,d_k \ll (R D q^2_\chi)^{2k}}
\frac{\prod_{1 \leq i < j \leq k} 1_{(d_i,d_j)|h_i-h_j}}{[d_1,\dots,d_k]} \frac{1}{x} \int_0^x \prod_{j=1}^k \Psi_{d_j}(y+h_j)\ dy
\approx {\mathfrak S}.$$
The contribution of those $y$ with $y \leq x^{1-\eps^2_0}$ is bounded by
$$ x^{-\eps^2_0} \log^{O(1)} x \sum_{d_1,\dots,d_k \ll R^2 (D q^2_\chi)^2}
\frac{\prod_{1 \leq i < j \leq k} 1_{(d_i,d_j)|h_i-h_j}}{[d_1,\dots,d_k]} \tau(d_1)^{O(1)} \cdots \tau(d_k)^{O(1)}.$$
Bounding $1_{(d_i,d_j)|h_i-h_j}{[d_1,\dots,d_k]} \ll \frac{1}{d_1 \cdots d_k}$ and using~\eqref{divisor-bound} we see that this contribution is $\approx 0$.  Thus it will suffice to establish the pointwise bound
\begin{equation}\label{dddr}
 \sum_{d_1,\dots,d_k \ll R^2 (D q^2_\chi)^2}
\frac{\prod_{1 \leq i < j \leq k} 1_{(d_i,d_j)|h_i-h_j}}{[d_1,\dots,d_k]} \prod_{j=1}^k \Psi_{d_j}(y+h_j)
\approx {\mathfrak S} 
\end{equation}
for all $x^{1-\eps^2_0} \leq y \leq x$.  Note that the restrictions on $d_1,\dots,d_k$ can be dropped thanks to the support of the $\Psi_{d_j}$.

By construction we have
\begin{equation}\label{psidy}
 \Psi_d(y+h_j) = \sum_{d_0,d_1,d_2: [d,d_1,d_2]=d} \chi(d_0) \mu(d_1) \mu(d_2) \Phi\left(\frac{y+h_j}{d_0}\right) \psi_{\leq R}(d_1) \psi_{\leq R}(d_2)
\end{equation}
This function is not multiplicative in $d$, but it can be Fourier expanded as a linear combination of multiplicative functions:

\begin{lemma}[Fourier expansion]  We have
\begin{equation}\label{psir-expand}
 \psi_{\leq R}(d) = \int_\R \frac{1}{d^{\frac{1+it}{\log R}}} f(t)\ dt
\end{equation}
and
\begin{equation}\label{phir-expand}
 \Phi\left(\frac{y+h_j}{d}\right) = \log x \int_\R \frac{1}{d^{\frac{1+it}{\log x}}} F_j(t)\ dt
\end{equation}
for all real $d \geq 1$, and some measurable functions $f,F_j \colon \R \to \C$ obeying the decay estimates
\begin{equation}\label{f-decay}
 f(t) \ll_m (1+|t|)^{-m}
\end{equation}
and
\begin{equation}\label{t-decay}
 F_j(t) \ll_m (1+|t|)^{-m} \
\end{equation}
for all $t \in \R$ and $m \geq 0$, as well as the identities
\begin{equation}\label{f-inc}
\int_\R f(t)\ dt = 1.
\end{equation}
and
\begin{equation}\label{Fj-ident}
 \int_\R F_j(t) (1+it)\ dt = 1.
\end{equation}
\end{lemma}

\begin{proof}  From~\eqref{psia} and~\eqref{fourier} we obtain~\eqref{psir-expand} with
$$ f(t) \coloneqq \frac{1}{2\pi} \int_\R e^{(1+it) u} \psi(u)\ du $$
the Fourier transform of $u \mapsto e^u \psi(u)$. From repeated integration by parts we have the rapid decrease~\eqref{f-decay}, while
from~\eqref{fourier-0} we have
$$
\int_\R f(t)\ dt = e^0 \psi(0) = 1$$
giving~\eqref{f-inc}.

Next, from~\eqref{naq} we have
$$ \Phi\left(\frac{y+h_j}{d}\right) = \int_0^{\infty} \psi\left(\frac{\log(x/t)}{2\log (Dq^2_\chi)}\right) \Phi_{dt}(y+h_j) \log t \frac{dt}{t}.$$
for any natural number $d$.  Writing $s = \frac{\log(x/t)}{\log x}$, we can rewrite this as
$$ \Phi\left(\frac{y+h_j}{d}\right) = \log^2 x \int_\R \psi\left(\frac{\log x}{2\log (Dq^2_\chi)} s\right) \phi\left(\log \frac{y+h_j}{x} + s \log x - \log d\right) (1 - s)\ ds.$$
By Fourier inversion~\eqref{fourier} we then have~\eqref{phir-expand}
where
$$ F_j(t) \coloneqq \frac{\log x}{2\pi} \int_\R e^{(1+it) u} \int_\R \psi\left(\frac{\log x}{2\log (Dq^2_\chi)} s\right) \phi\left(\log \frac{y+h_j}{x} + (s - u) \log x\right) (1 - s)\ ds du,$$
which on making the change of variables $v \coloneqq u-s$ factors as
$$ F_j(t) = \frac{\log x}{2\pi} \left(\int_\R e^{(1+it) v} \phi\left(\log \frac{y+h_j}{x} - v \log x\right) \ dv\right) \left(\int_\R e^{(1+it) s} \psi\left(\frac{\log x}{2\log (Dq^2_\chi)} s\right) (1-s)\ ds\right).$$
From the triangle inequality one has
$$ \int_\R e^{(1+it) v} \phi\left(\log \frac{y+h_j}{x} - v \log x\right) \ dv \ll \frac{1}{\log x}$$
while from integration by parts (and~\eqref{D-def}) one has
$$ \int_\R e^{(1+it) s} \psi\left(\frac{\log x}{2\log (Dq^2_\chi)} s\right) (1-s)\ ds \ll_m (1+|t|)^{-m}$$
for any $m \geq 0$, thus yielding~\eqref{t-decay}.  Also, from~\eqref{fourier-1}, and integration by parts one has
\begin{align*}
 \int_\R F_j(t) (1+it)\ dt &= - \frac{d}{dx} \Phi\left(\frac{y+h_j}{x}\right)|_{x=0} \\
&= (y+h_j) \Phi'(y+h_j) \\
&= \int_0^{\infty} \psi\left(\frac{\log(x/t)}{2\log (Dq^2_\chi)}\right) \phi'\left(\log \frac{y+h_j}{t}\right) \log t \frac{dt}{t} \\
&=  \int_0^{\infty} \phi'\left(\log \frac{y+h_j}{t}\right) \log t \frac{dt}{t} \\
&= \int_\R \phi'(u) (\log(y + h_j) - u)\ du \\
&= \int_\R \phi(u)\ du \\
&= 1
\end{align*}
where we have used the observation that $\psi(\frac{\log(x/t)}{2\log (Dq^2_\chi)})$ equals to $1$ on the support of $\phi'(\log \frac{y+h_j}{t})$ (since one then has $x/t \asymp x/y \ll x^{\eps^2_0}$). This gives~\eqref{Fj-ident}.
\end{proof}

Inserting the expansions~\eqref{psir-expand}, \eqref{phir-expand} back into~\eqref{psidy}, we see that
$$
 \Psi_d(y+h_j) = \log^k x \int_\R \int_\R \int_\R \sum_{d_0,d_1,d_2: [d,d_1,d_2]=d} \frac{\chi(d_0) \mu(d_1) \mu(d_2)}{d_0^{\frac{1+it_0}{\log x}} d_1^{\frac{1+it_1}{\log R}} d_2^{\frac{1+it_2}{\log R}}}\ F_j(t_0) f(t_1) f(t_2) dt_0 dt_1 dt_2.$$
Inserting this back into the left-hand side of~\eqref{dddr} and factoring the Euler product using~\eqref{euler}, we can thus write that left-hand side as
\begin{equation}\label{ep-int}
 \log^k x \int_{\R^{3k}} \prod_p E_{p,t_{0,1},\dots,t_{2,k}}\ \prod_{j=1}^k F_j(t_{0,j}) f(t_{1,j}) f(t_{2,j}) dt_{0,j} dt_{1,j} dt_{2,j},
\end{equation}
where
\begin{equation}\label{ept}
E_{p,t_{0,1},\dots,t_{2,k}} \coloneqq \sum_{d_{1},\dots,d_{k} \in \N_{(p)}}
\frac{\prod_{1 \leq i < j \leq k} 1_{(d_i,d_j)|h_i-h_j}}{[d_1,\dots,d_k]} \prod_{j=1}^k c_{d_j, t_{0,j}, t_{1,j}, t_{2,j}},
\end{equation}
and
\begin{equation}\label{cdt-def}
 c_{d,t_0,t_1,t_2} \coloneqq \sum_{d_0,d_1,d_2: [d_0,d_1,d_2]=d}
\frac{\chi(d_{0}) \mu(d_{1}) \mu(d_{2})}{d_{0}^{\frac{1+it_{0}}{\log x}} d_{1}^{\frac{1+it_{1}}{\log R}} d_{2}^{\frac{1+it_{2}}{\log R}}}.
\end{equation}
From the triangle inequality one has the crude bound
\begin{equation}\label{log}
 E_{p,t_{0,1},\dots,t_{2,k}} = 1 + O\left( \frac{1}{p^{1+1/\log R}} \right)
\end{equation}
and thus by Mertens' theorem~\eqref{mertens-4}
$$ \prod_p E_{p,t_{0,1},\dots,t_{2,k}} \ll \log^{O(1)} R.$$
Using~\eqref{f-decay}, \eqref{t-decay} we see that the contribution of the integral in which the quantity
$$ |t| \coloneqq \sup_{0 \leq i \leq 2; 1 \leq j \leq k} |t_{i,j}| $$
exceeds (say) $\log^{1/10} R$ is negligible.  Thus we may restrict attention to the regime
$$
 |t| \leq \log^{1/10} R.
$$
We can improve the above analysis to restrict the region of $t$ further.  From Taylor expansion one has the more precise bound
$$ E_{p,t_{0,1},\dots,t_{2,k}} = 1 - \frac{k}{p} + O\left( \frac{(1+|t|)^3 \log_R p}{p} \right) + O( \frac{1}{p^2} )$$
when $p \leq R$.  Using this bound in place of~\eqref{log} when $\log p \leq (1+|t|)^{-3} \log R$ and using Mertens' theorem~\eqref{mertens}, \eqref{mertens-3}, we obtain the refined estimate
\begin{equation}\label{ppc}
 \prod_{p \geq C} E_{p,t_{0,1},\dots,t_{2,k}} \ll_C (1+|t|)^{O(1)} \log^{-k} R
\end{equation}
for any $C \geq 1$.
Using~\eqref{f-decay}, \eqref{t-decay} we see that the contribution of the integral in which $|t| \geq \log^{1/(100k)} \eta$ (say) is negligible.  Thus we may restrict attention to the regime
\begin{equation}\label{tank}
 |t| \leq \log^{1/(100k)} \eta.
\end{equation}

We now perform an even more precise analysis of the Euler factors $E_{p,t_{0,1},\dots,t_{2,k}}$.
Let us first suppose that $p$ is larger than $C_0$ for some sufficiently large $C_0$ (depending on $h_1,\dots,h_k,k$).  Then $p$ does not divide $\prod_{1 \leq i < j \leq k}(h_i-h_j)$.  Thus in order for the sum in~\eqref{ept} to be non-zero, at most one of the $d_j$ can be greater than $1$, and hence
$$
 E_{p,t_{0,1},\dots,t_{2,k}} = 1 + \sum_{j=1}^k \sum_{l=1}^\infty \frac{c_{p^l, t_{0,j}, t_{1,j}, t_{2,j}}}{p^l}.$$
We expand $c_{p^l, t_{0,j}, t_{1,j}, t_{2,j}}$ using~\eqref{cdt-def}.  For $l>1$, the sum in~\eqref{cdt-def} only consists of those terms with $d_0=p^l$, and thus 
$$ c_{p^l, t_{0,j}, t_{1,j}, t_{2,j}} = \frac{\chi(p^l)}{p^{l(1+\frac{1+it_{0,j}}{\log x})}}
\left( 1 - \frac{1}{p^{\frac{1+it_{1,j}}{\log R}}} \right) \left( 1 - \frac{1}{p^{\frac{1+it_{2,j}}{\log R}}} \right).$$
In particular, from Taylor expansion, we have
\begin{equation}\label{cpl}
 c_{p^l, t_{0,j}, t_{1,j}, t_{2,j}}  \ll \min((1+|t|) \log_R p,1)^2 ).
\end{equation}
For $l=1$, the sum in~\eqref{cdt-def} consists of those terms with $d_0,d_1,d_2 \in \{1,p\}$, excluding the triple $d_0=d_1=d_2=0$, thus 
$$ c_{p,t_0,t_1,t_2} = \left( 1 + \frac{\chi(p)}{p^{\frac{1+it_{0}}{\log x}}} \right)
\left( 1 - \frac{1}{p^{\frac{1+it_{1}}{\log R}}} \right) \left( 1 - \frac{1}{p^{\frac{1+it_{2}}{\log R}}} \right) - 1.$$
We thus have
\begin{align*}
 E_{p,t_{0,1},\dots,t_{2,k}} &= 1 + \sum_{j=1}^k \frac{c_{p,t_{0,j},t_{1,j},t_{2,j}}}{p} 
\left( 1 - \frac{1}{p^{\frac{1+it_{1,j}}{\log R}}} \right) \left( 1 - \frac{1}{p^{\frac{1+it_{2,j}}{\log R}}} \right) - \frac{1}{p} \\
&\quad + O( \frac{\min((1+|t|)\log_R p,1)^2 }{p^2} ).
\end{align*}
For $p \geq C_0$, we may use the trivial bound $c_{p,t_{0,j},t_{1,j},t_{2,j}} \ll 1$ to factor
$$
 E_{p,t_{0,1},\dots,t_{2,k}} = \left(1 + \sum_{j=1}^k \frac{c_{p,t_{0,j},t_{1,j},t_{2,j}}}{p}\right) \exp\left(O\left( \frac{\min((1+|t|)\log_R p,1)^2 }{p^2} \right) \right).
$$
Since
$$ \sum_p \frac{\min((1+|t|)\log_R p,1)^2 }{p^2} \ll
\frac{(1+|t|)^2}{\log^2 R} \sum_p \frac{1}{p^{3/2}} \ll \frac{1}{\log R} $$
we thus have
$$\prod_{p \geq C_0} E_{p,t_{0,1},\dots,t_{2,k}} = \exp\left( O\left( \frac{1}{\log R} \right) \right) \prod_{p \geq C_0} \left(1 + \sum_{j=1}^k \frac{c_{p,t_{0,j},t_{1,j},t_{2,j}}}{p}\right) .$$
Let us compare $c_{p,t_{0,j},t_{1,j},t_{2,j}}$ against the quantity
$$ c'_{p,t_{0,j},t_{1,j},t_{2,j}} \coloneqq \left( 1 - \frac{1}{p^{\frac{1+it_{0,j}}{\log x}}} \right)
\left( 1 - \frac{1}{p^{\frac{1+it_{1,j}}{\log R}}} \right) \left( 1 - \frac{1}{p^{\frac{1+it_{2,j}}{\log R}}} \right) - 1.$$
The two quantities agree unless $p$ is exceptional.  From the triangle inequality we have the crude bound
$$ c_{p,t_{0,j},t_{1,j},t_{2,j}} - c'_{p,t_{0,j},t_{1,j},t_{2,j}} \ll \frac{1}{p^{\frac{1}{\log x}}},$$
and when $p \leq x$ we can use Taylor expansion and~\eqref{tank} to also obtain the bound
$$ c_{p,t_{0,j},t_{1,j},t_{2,j}} - c'_{p,t_{0,j},t_{1,j},t_{2,j}} \ll (\log^{1/(100k)} \eta \log_R p)^2.$$
Thus in all cases one has the bound
\begin{equation}\label{cpl-1}
 c_{p,t_{0,j},t_{1,j},t_{2,j}} - c'_{p,t_{0,j},t_{1,j},t_{2,j}} \ll 1_{p\textnormal{ exceptional }} \frac{(\log^{1/(100k)} \eta \log_R p)^2}{p^{\frac{1}{\log x}}}.
\end{equation}
Applying Corollary~\ref{excep-bound}, we conclude that
$$\prod_{p \geq C_0} E_{p,t_{0,1},\dots,t_{2,k}} = \exp\left( O\left( \frac{1}{\log^{1/(7k)} \eta} \right) \right) \prod_{p \geq C_0} \left(1 + \sum_{j=1}^k \sum_{p \geq C_0} \frac{c'_{p,t_{0,j},t_{1,j},t_{2,j}}}{p} \right).$$
Bounding $\frac{1}{p^{\frac{1+it_{1,j}}{\log R}}}, \frac{1}{p^{\frac{1+it_{2,j}}{\log R}}} = O( \exp(-\log_R p) )$ we have
$$ c'_{p,t_{0,j},t_{1,j},t_{2,j}}  = - \frac{1}{p^{\frac{1+it_{0,j}}{\log x}}}  + O( \exp(-2\log_R p) ),$$
while from Taylor expansion we see for $p \leq x$ that
\begin{equation}\label{cpot}
\begin{split}
 c'_{p,t_{0,j},t_{1,j},t_{2,j}} &= -1 + O( (1 + |t|) \log_x p )\\
&=  - \frac{1}{p^{\frac{1+it_{0,j}}{\log x}}}  + O( (1 + |t|) \log_x p ) \\
&=  - \frac{1}{p^{\frac{1+it_{0,j}}{\log x}}}  + O\left( \frac{1}{\log^{1/(6k)} \eta} \log_R p \right)
\end{split}
\end{equation}
thanks to~\eqref{tank}, \eqref{R-def}.
Combining the bounds, we see that
$$ c'_{p,t_{0,j},t_{1,j},t_{2,j}} =  - \frac{1}{p^{\frac{1+it_{0,j}}{\log x}}}  + O\left( \min\left( \frac{\log_R p}{\log^{1/(6k)} \eta}, \exp( - 2\log_R p ) \right) \right)
$$
for all $p \geq C_0$.
From Mertens' theorem (\eqref{mertens} for $\log_R p \leq \log^{1/(100k)} \eta$ and~\eqref{mertens-2} for $\log_R p > \log^{1/(100k)} \eta$) we have
$$
\sum_p \frac{\min\left( \frac{1}{\log^{1/(6k)} \eta} \log_R p, \exp( - 2\log_R p ) \right) }{p} \ll \frac{1}{\log^{1/(7k)} \eta}$$
(say). We conclude that
$$\prod_{p \geq C_0} E_{p,t_{0,1},\dots,t_{2,k}} = \exp\left( O\left( \frac{1}{\log^{1/(7k)} \eta} \right) \right)
\prod_{p \geq C_0} \left(1 - \sum_{j=1}^k \frac{1}{p^{1+\frac{1+it_{0,j}}{\log x}}}\right).$$
The function
$$ \prod_{p \geq C_0} \frac{1 - \sum_{j=1}^k \frac{1}{p^{s_j}}}{\prod_{j=1}^k (1 - \frac{1}{p^{s_j}})}$$
converges to a holomorphic function of $s_1,\dots,s_k$ in the polydisk $\prod_{j=1}^k \{ s_j: |s_j-1| \leq 1/2\}$ which is bounded in magnitude by $O(1)$ (since each factor here is $1+O(1/p^2)$).  From the Cauchy integral formula we conclude that
$$ \prod_{p \geq C_0} \frac{1 - \sum_{j=1}^k \frac{1}{p^{s_j}}}{\prod_{j=1}^k (1 - \frac{1}{p^{s_j}})}
= \prod_{p \geq C_0} \frac{1 - \sum_{j=1}^k \frac{1}{p}}{\prod_{j=1}^k (1 - \frac{1}{p})} ( 1 + \max(|s_1-1|, \dots,|s_k-1|) )$$
when $|s_1-1|,\dots,|s_k-1| \leq \frac{1}{4}$.  Observing from~\eqref{betap-def} that
$$ \frac{1 - \sum_{j=1}^k \frac{1}{p}}{\prod_{j=1}^k (1 - \frac{1}{p})} = \beta_p$$
for $p \geq C_0$, we conclude (using~\eqref{tank}) that
$$ \prod_{p \geq C_0} \frac{1 - \sum_{j=1}^k \frac{1}{p^{1+\frac{1+it_{0,j}}{\log x}}}}{\prod_{j=1}^k \left(1-\frac{1}{p^{1+\frac{1+it_{0,j}}{\log x}}}\right)} = \exp\left( O\left( \frac{1+|t|}{\log x} \right) \right) \prod_{p \geq C_0} \beta_p$$
and thus (by~\eqref{tank}, \eqref{R-def})
\begin{equation}\label{cake}
\prod_{p \geq C_0} E_{p,t_{0,1},\dots,t_{2,k}} = \exp\left( O\left( \frac{1}{\log^{1/(7k)} \eta} \right) \right)
\prod_{p \geq C_0} \beta_p \prod_{j=1}^k \left(1-\frac{1}{p^{1+\frac{1+it_{0,j}}{\log x}}}\right).
\end{equation}
Now we turn attention to the small primes $p<C_0$.  Using~\eqref{ept}, \eqref{cpl} we have
$$ E_{p,t_{0,1},\dots,t_{2,k}} = \sum_{d_{1},\dots,d_{k} \in \{1,p\}}
\frac{\prod_{1 \leq i < j \leq k} 1_{(d_i,d_j)|h_i-h_j}}{[d_1,\dots,d_k]} \prod_{j=1}^k c_{d_j, t_{0,j}, t_{1,j}, t_{2,j}}
+ O\left(\left ( (1+|t|) \frac{\log C_0}{\log R} \right)^2 \right)$$
for $p \leq C_0$, which we rewrite as
\begin{align*} E_{p,t_{0,1},\dots,t_{2,k}} &= 1 + \sum_{d_{1},\dots,d_{k} \in \{1,p\}: [d_1,\dots,d_k]=p}
\frac{\prod_{1 \leq i < j \leq k} 1_{(d_i,d_j)|h_i-h_j}}{p} \prod_{j=1}^k c_{d_j, t_{0,j}, t_{1,j}, t_{2,j}}\\
&\quad \quad + O\left( \left( (1+|t|) \frac{\log C_0}{\log R} \right)^2 \right)
\end{align*}
From~\eqref{tank}, \eqref{R-def} the error term is certainly $O( \frac{1}{\log^{1/(7k)} \eta} )$.  From~\eqref{cpl-1}, \eqref{cpot}, \eqref{tank}, \eqref{R-def} we similarly have
$$ c_{p,t_{0,j},t_{1,j},t_{2,j}} = -1 + O\left(\frac{1}{\log^{1/(7k)} \eta}\right)$$
for $p \leq C_0$, and thus
$$ c_{d_j,t_{0,j},t_{1,j},t_{2,j}} = \mu(d_j) + O\left(\frac{1}{\log^{1/(7k)} \eta}\right)$$
for $j=0,\dots,k$.  This gives
$$ E_{p,t_{0,1},\dots,t_{2,k}} = 1 + \sum_{d_{1},\dots,d_{k} \in \{1,p\}: [d_1,\dots,d_k] = p}
\mu(d_1) \cdots \mu(d_k) \frac{\prod_{1 \leq i < j \leq k} 1_{(d_i,d_j)|h_i-h_j}}{p} + O\left( \frac{1}{\log^{1/(7k)} \eta} \right).$$
If the $h_i$ occupy $m$ distinct residue classes $b_1,\dots,b_m$ modulo $p$, then the constraint $\prod_{1 \leq i < j \leq k} 1_{(d_i,d_j)|h_i-h_j}$ constrains the index set $\{ i: d_i = p\}$ to be a subset of one of the sets $\{ i: h_i = b_j\ (p)\}$ for $j=1,\dots,m$, which must be non-empty if $[d_1,\dots,d_k]$ is to equal $p$.  From the alternating sign of the M\"obius function, each $j$ has a net contribution of $-\frac{1}{p}$ to the above sum, thus
$$ E_{p,t_{0,1},\dots,t_{2,k}} = 1 - \frac{m}{p} + O\left( \frac{1}{\log^{1/(7k)} \eta} \right).$$
From~\eqref{betap-def} we have
$$ \beta_p = \left(1-\frac{m}{p}\right) \left(1-\frac{1}{p}\right)^{-k}$$
and thus
$$ E_{p,t_{0,1},\dots,t_{2,k}} = \left(1-\frac{1}{p}\right)^k \beta_p + O\left( \frac{1}{\log^{1/(7k)} \eta} \right);$$
by Taylor expansion and~\eqref{tank} we then have
$$ E_{p,t_{0,1},\dots,t_{2,k}} = \beta_p \prod_{j=1}^k \left(1-\frac{1}{p^{1+\frac{1+it_{0,j}}{\log x}}}\right) + O\left( \frac{1}{\log^{1/(7k)} \eta} \right)$$
for $p < C_0$.
If we now fix $C_0$ so that all the previous estimates are justified, we have
$$ \prod_{p < C_0} E_{p,t_{0,1},\dots,t_{2,k}}  = \prod_{p < C_0} \beta_p \prod_{j=1}^k \left(1-\frac{1}{p^{1+\frac{1+it_{0,j}}{\log x}}}\right)  + O\left( \frac{1}{\log^{1/(7k)} \eta} \right)$$
and hence by~\eqref{cake}, \eqref{mertens-2}, \eqref{ss-def}
\begin{align*}
\prod_{p} E_{p,t_{0,1},\dots,t_{2,k}}  &= {\mathfrak S} \prod_{p} \prod_{j=1}^k \left(1-\frac{1}{p^{1+\frac{1+it_{0,j}}{\log x}}}\right) + O\left( \frac{1}{\log^{1/(7k)} \eta} \prod_{p}\left(1-\frac{1}{p^{1+\frac{1}{\log x}}}\right)^k\right)\\
&= {\mathfrak S} \prod_{p} \prod_{j=1}^k \left(1-\frac{1}{p^{1+\frac{1+it_{0,j}}{\log x}}}\right) + O\left(\frac{\log^{-k}x}{\log^{1/(7k)} \eta}\right).
\end{align*}
From the Euler product formula~\eqref{prod-p} as well as~\eqref{tank}, we conclude that
$$ \prod_{p} E_{p,t_{0,1},\dots,t_{2,k}}  = {\mathfrak S} \log^{-k} x \prod_{j=1}^k (1+t_{0,j}) + O\left( \frac{\log^{-k} x}{\log^{1/(7k)} \eta} \right).$$
Inserting this bound into~\eqref{ep-int}, and using~\eqref{f-decay}, \eqref{t-decay} to remove the restriction~\eqref{tank}, we can thus write the left-hand side of~\eqref{dddr} as
$$ \approx {\mathfrak S} \int_{\R^{3k}} \prod_{j=1}^k (1+t_{0,j}) F_j(t_{0,j}) f(t_{1,j}) f(t_{2,j})\ dt_{0,j} dt_{1,j} dt_{2,j}.$$
Applying~\eqref{f-inc}, \eqref{Fj-ident}, this is $\approx {\mathfrak S}$, giving the claim.  This (finally!) concludes the proof of Theorem~\ref{main}.

\bibliography{siegelrefs}
\bibliographystyle{plain}

\end{document}